\documentclass[reqno]{amsart}

\usepackage{amsfonts}

\usepackage{amsmath,amssymb,amsthm,amsfonts}
\usepackage{mathrsfs}
\usepackage{bbm}
\usepackage{hyperref}

\usepackage{esint}

\usepackage{color}

\newtheorem{lemma}{Lemma}[section]
\newtheorem{theorem}{Theorem}[section]

\newtheorem{proposition}{Proposition}[section]

\usepackage{mathtools}
\numberwithin{equation}{section}

\newcommand{\be}{\begin{equation}}
\newcommand{\bm}{\begin{multline}}
\newcommand{\ee}{\end{equation}}

\newcommand{\dis}{\displaystyle}

\newcommand{\R}{\mathbb{R}}

\renewcommand{\S}{\mathbb{S}}
\newcommand{\T}{\mathbb{T}}

\newcommand{\CF}{\mathcal{F}}

\newcommand{\ga}{\gamma}

\newcommand{\de}{\delta}
\newcommand{\si}{\sigma}

\newcommand{\Ga}{\Gamma}

\makeatletter
\@namedef{subjclassname@2020}{%
  \textup{2020} Mathematics Subject Classification}
\makeatother

\begin{document}

\title[An $L^1_k\cap L^p_k$ approach for the non-cutoff Boltzmann equation in $\R^3$]{An $L^1_k\cap L^p_k$ approach for the non-cutoff Boltzmann equation in $\R^3$}

\author[R.-J. Duan]{Renjun Duan}
\address[RJD]{Department of Mathematics, The Chinese University of Hong Kong,
Shatin, Hong Kong, P.R.~China}
\email{rjduan@math.cuhk.edu.hk}

\author[S. Sakamoto]{Shota Sakamoto}
\address[SS]{Department of Mathematics,
Tokyo Institute of Technology,
Tokyo, 980-0856, Japan}
\email{sakamoto.s.aj@m.titech.ac.jp}

\author[Y. Ueda]{Yoshihiro Ueda}
\address[YU]{Graduate School of Maritime Sciences,
Kobe University,
Kobe, 658-0022, Japan}
\email{ueda@maritime.kobe-u.ac.jp}

\begin{abstract}
In the paper, we develop an $L^1_k\cap L^p_k$ approach to construct global solutions to the Cauchy problem on the non-cutoff Boltzmann equation near equilibrium in $\R^3$. In particular, only smallness of $\|\CF_x{f}_0\|_{L^1\cap L^p (\R^3_k;L^2(\R^3_v))}$ with $3/2<p\leq \infty$ is imposed on initial data $f_0(x,v)$, where $\CF_x{f}_0(k,v)$ is the Fourier transform in space variable. This provides the first result on the global existence of such low-regularity solutions without relying on Sobolev embedding $H^2(\R^3_x)\subset L^\infty(\R^3_x)$ in case of the whole space.    Different from the use of sufficiently smooth Sobolev spaces in those classical results \cite{GS} and \cite{AMUXY-2012-JFA}, there is a crucial difference between the torus case and the whole space case for low regularity solutions under consideration. In fact, for the former, it is enough to take the only $L^1_k$ norm corresponding to the Weiner space as studied in \cite{DLSS}. In contrast, for the latter, the extra interplay with the $L^p_k$ norm plays a vital role in controlling the nonlinear collision term due to the degenerate dissipation of the macroscopic component. Indeed, the propagation of $L^p_k$ norm helps gain an almost optimal decay rate 
$ (1+t)^{-\frac{3}{2} (1-\frac{1}{p})_+}$ of the $L^1_k$ norm via the time-weighted energy estimates in the spirit of the idea of \cite{KNN} and in turn, 
this is necessarily used for establishing the global existence.
\end{abstract}

\date{\today}

\maketitle

\thispagestyle{empty}
\setcounter{tocdepth}{3}
\tableofcontents

\section{Introduction}

We consider the Cauchy problem on the non-cutoff Boltzmann equation in the whole space 
\begin{align}\label{eq: Boltzmann}
\begin{cases}
\partial_t F+v\cdot \nabla_x F=Q(F,F), 
\\
F(0,x,v)=F_0(x,v).
\end{cases}
\end{align}
Here the unknown $F=F(t,x,v)\geq 0$ denotes the density distribution function of gas particles with position  $x\in \mathbb{R}^3$ and velocity $v\in \mathbb{R}^3$ at time $t\geq 0$ and $F_0(x,v)$ is the given initial data. The bilinear collision operator $Q$ is defined as
\begin{align*}
Q(G,F)(v)=\int_{\R^3}\int_{\S^2}B(v-u,\si)
\left[G(u')F(v')-G(u)F(v)\right]\,d\si
d u,
\end{align*}
where the velocity pairs $(v,u)$ and $(v',u')$  satisfy
\begin{align*}
v'=\frac{v+u}{2}+\frac{|v-u|}{2}\sigma,\
u'=\frac{v+u}{2}-\frac{|v-u|}{2}\sigma,\quad \sigma\in \mathbb{S}^2.
\end{align*}
We assume that the collision kernel $B$ takes the form
\begin{equation}\label{ass.ck}
\left\{\begin{aligned}
B(v-u,\sigma)&=\vert v-u\vert^\gamma b(\cos\theta),\ & \gamma\in (-3,1],\\
\cos \theta&=\frac{v-u}{\vert v-u\vert}\cdot \sigma, \ & \theta \in (0,\pi/2],\\
\sin \theta b(\cos\theta)&\simeq \theta^{-1-2s}\ \text{as }\theta\to 0,
\ & 0<s<1.
\end{aligned}\right.
\end{equation}
We note that $Q$ has been symmetrized in a standard way such that $B$ is supported in $0<\theta\leq \pi/2$, cf.~\cite{Gla, Vil98}.
In this paper, we are concerned with the hard and moderately soft potential cases with a restriction
\begin{align}\label{ass.add}
\gamma>\max\{ -3, -3/2-2s\},
\end{align}
that is needed only in the use of a technical trilinear estimate \eqref{ineq:Trilinear}, cf.~\cite[Theorem 1.2]{AMUXY-KRM-2013}.

Let $\mu=\mu(v):=(2\pi)^{-3/2}e^{-\vert v\vert^2/2}$ be a normalized global Maxwellian as a reference equilibrium. We insert $F(t,x,v)=\mu+\mu^{1/2}f(t,x,v)$ into \eqref{eq: Boltzmann} and derive the reformulated Cauchy problem on $f=f(t,x,v)$ as follows
\begin{align}\label{BE}
\begin{cases}
\partial_t f+v\cdot \nabla_x f+Lf=\Gamma(f,f),\\
f(0,x,v)=f_0(x,v),
\end{cases}
\end{align}
where the linearized collision operator $L$ and the nonlinear operator $\Gamma$ are respectively given by 
\begin{align}
Lf&=-\mu^{-1/2}[Q(\mu, \mu^{1/2}f)+ Q(\mu^{1/2}f, \mu)],
\label{LinearOp}\\
\Gamma(f,f)&=\mu^{-1/2}Q(\mu^{1/2}f, \mu^{1/2}f).\notag
\end{align}

Let us define some norms to describe a preceding result on a solution of \eqref{BE} constructed in terms of the Wiener space when the space domain is a torus, and also introduce some other norms to further present our main results of this paper afterward. First, let  $\Vert \cdot \Vert_{L^p}$ denote the usual $L^p$-norm.
We also use a weighted $L^2$ space endowed with the norm
\begin{align*}
\Vert f\Vert_{L^2_\alpha}^2=\int_{\mathbb{R}^3} \vert \langle v\rangle^\alpha f(v)\vert^2 dv,\quad 
 \langle v\rangle:=\sqrt{1+|v|^2},
\end{align*}
for $\alpha\in\mathbb{R}$. The Bochner spaces such as $L^p_t L^q_v$ are endowed with the norm $\Vert \Vert \cdot \Vert_{L^q_v} \Vert_{L^p_t}$. Let $\Vert \cdot \Vert_{L^2_{v,D}}$ be the energy dissipation norm defined by \cite{AMUXY-2011-CMP}, which takes the form
\begin{align*}
\Vert f \Vert_{L^2_{v,D}}^2&=\iiint_{\mathbb{R}^3\times \mathbb{R}^3\times \mathbb{S}^2} B \mu(u) (f(v')-f(v))^2 dvdu d\sigma \\
&+\iiint_{\mathbb{R}^3\times \mathbb{R}^3\times \mathbb{S}^2} B f(u)^2 (\mu(v')^{1/2}-\mu(v)^{1/2})^2 dvdu d\sigma.
\end{align*}
By $L^p_k$ it means that we take the $L^p$-norm on the Fourier side. Especially we use the following norms
\begin{align*}
\Vert f\Vert_{L^1_k L^\infty_T L^2_v} &=\int_{\hat{\Omega}_k} \sup_{0\le t\le T} \|\hat{f}(t,k,\cdot)\|_{L^2_v} dk,\\
\Vert f\Vert_{L^1_k L^2_T L^2_{v,D}} &=\int_{\hat{\Omega}_k}  \Big(\int^T_0   \Vert \hat{f}(t,k,\cdot)\Vert_{L^2_{v,D}}^2  dt\Big)^{1/2} dk,
\end{align*}
for any $T>0$,  where we have defined the usual Fourier transform $\hat{f}(t,k,v)=\mathcal{F}_x [f(t,\cdot,v)](k)$ with $\hat{\Omega}_k=\mathbb{R}^3_k$ if $\Omega=\mathbb{R}^3_x$ while $\hat{\Omega}_k=\mathbb{Z}^3_k$ if $\Omega=\mathbb{T}^3$, and $dk$ denotes the corresponding Lebesgue measure on each space. Note that in the case of a torus,  $L^1_k$ is the so-called Wiener space, usually denoted by $A(\mathbb{T}^3)$ in the literature, cf.~\cite[page 200]{Gra}.

For the case, $\Omega=\mathbb{T}^3$, the first and second authors with their collaborators have obtained the unique global existence of a solution with an exponential convergence rate. To be precise, we state part of the corresponding results as follows.

\begin{proposition}[\cite{DLSS}]\label{DLSS Existence}
Let $\Omega=\mathbb{T}^3$.  
Let a weight function $w_{q}$ be defined as
\begin{align*}
w_{q}(v)
=
\begin{cases}
1 \quad &(\gamma+2s\ge 0),\\
\displaystyle \exp\Big( \frac{q}{4}\langle v\rangle\Big),\ q>0 \quad & (\gamma+2s<0,\ \gamma>\max\{-3, -3/2-2s\}).
\end{cases}
\end{align*}
Assume that $f_0(x,v)$ satisfies the conservation laws of mass, 
momentum, 
and energy, namely, $\int_{\T^3}\int_{\R^3}\phi(v)f_0(x,v)dvdx=0$ for $\phi(v)=1,v,|v|^2$. 
Then there is $\epsilon_0>0$  such that if $F_0(x,v)=\mu+\mu^{\frac{1}{2}}f_0(x,v)\ge 0$ and
\begin{align*}
\Vert  w_q  f_0\Vert_{L^1_k L^2_v} \le \epsilon_0,
\end{align*}
then there exists a unique global mild solution $f(t,x,v)$, $t>0$, $x\in \mathbb{T}^3$, $v\in \mathbb{R}^3$ to the Cauchy problem 
\eqref{BE}, satisfying that  $F(t,x,v)=\mu+\mu^{\frac{1}{2}}f(t,x,v)\ge 0$ and
\begin{align*}
\| w_q f\|_{L^1_kL^\infty_TL^2_v}+\| w_q f\|_{L^1_kL^2_TL^2_{v,D}}\lesssim
\Vert  w_q f_0\Vert_{L^1_k L^2_v}
\end{align*}
for any $T>0$.
Moreover, let
\begin{align*}
\kappa=
\begin{cases}
1&(\ga+2s\geq 0),\\
\dis \frac{1}{1+|\ga+2s|} &(\gamma+2s<0,\ \gamma>\max\{-3, -3/2-2s\}),
\end{cases}
\end{align*}
then there is $\lambda>0$ such that the solution also enjoys the time decay estimate
\begin{align*}
\Vert w_q f(t)\Vert_{L^1_k L^2_v} \lesssim e^{-\lambda t^\kappa} \Vert w_q  f_0\Vert_{L^1_k L^2_v}
\end{align*}
for any $t\geq 0$. 
\end{proposition}

In this paper, we are interested in the case $\Omega=\mathbb{R}^3$. Recall that if $\Omega=\mathbb{T}^3$, one can first obtain the first order macro dissipation over the low-frequency domain and then apply the Poincar\'e inequality to recover the zero-order macro dissipation, 
which is crucially used to control the macro part of nonlinear terms, so that we can close the $L^1_k$ estimates by themselves. However, this is evidently not true if $\Omega=\mathbb{R}^3$. Therefore we have to find out an alternative way to close the a priori estimates on the nonlinear term for such low regularity solutions without using any explicit space derivative.

In order to overcome this difficulty, we are strongly motivated by Kawashima-Nishibata-Nishikawa \cite{KNN}, where the authors developed a new $L^p$ $(1\leq p\leq \infty)$ energy method for multi-dimensional viscous conservation laws in $\R^d$ with $d\geq 1$ for the initial data in $L^1\cap L^p$ that is even not necessarily small. Their energy method is useful enough to derive the optimal decay estimates of solutions in the $W^{1,p}$ space for the Cauchy problem.
Moreover, one advantage of this method is that for the $L^p$-$L^1$ time decay estimates one does not need any explicit representation of the semigroup via the linearized Fourier analysis but rather uses the time-weighted method with the help of interpolation inequalities. 

To incorporate the above idea into the study of our problem, we first derive the uniform $L^1_k\cap L^p_k$ estimates under the a priori assumption of both smallness and time decay of the macro part. Note that the time decay of the macro part should be strong enough to assure the convergence of the $L^2_T$ norm. Then it suffices to derive the time decay of solutions to close the a priori assumption. Since we will work on the Fourier side, the interpolation inequalities we use for these estimates are given in terms of the Fourier variable $k$. We develop those interpolation inequalities both in time-frequency double variables for hard potentials and in time-frequency-velocity triple variables for soft potentials.
These estimates indeed prove that the a priori assumption can be verified if initial data is sufficiently small in terms of the $L^1_k\cap L^p_k$ norm. This is our strategy in this paper and more details will be given at the end of Section \ref{sec2}.

Before we introduce our main results in the next section, let us carry out a short review on the Boltzmann equation in the perturbation framework near global Maxwellians. For brevity, we mainly focus on only the time-decay of solutions. The readers may find a detailed review on global existence problems of such equations in the recent work \cite{DLSS}. In particular, we only mention the first results on the non-cutoff Boltzmann equation independently by Gressman-Strain \cite{GS} and AMUXY \cite{AMUXY-2011-CMP,AMUXY-2012-JFA}. 

Along with the existence problem, the decay rate of solutions to kinetic equations, such as the Boltzmann equation, is also a long-standing problem of which we cannot explain all the related literature here. Therefore we will briefly review some of them. Readers may find works on decay rate for the angular cut-off case in \cite{Gla} and reference therein.

Ukai~\cite{Uk} first proved the decay of a solution (alongside its existence) on the torus. In that paper, spectral analysis is the key method. Decay on the torus is also shown in, including related equations,  \cite{Ca-1980, GS-TTSP-1999, DV,SG-CPDE,SG-08-ARMA, GS, BS-JFA-2013, DLSS}. For this case, the $H$-theorem is anticipated to dominate the system in a box, and thus solutions are shown to have the  (sub-)exponential decay to Maxwellians. 

On the whole space, the decay rate is expected to be worse than that on the torus because dispersivity coming from $\partial_t + v\cdot \nabla_x$ is dominant, and the optimal decay of linearized equations is a heat-like one. We list some of the results on the cut-off case first. Kawashima~\cite{K} obtained the optimal decay in $L^2_vH^\ell_x$, where $H^\ell$ is the standard $L^2$-based Sobolev space, by using the compensation function, constructed by the observation of thirteen moments. This method is improved by \cite{D, DSt} to use the macro-micro decomposition instead, which is now standard. Ukai-Yang~\cite{UY} showed the optimal decay in the space $L^2\cap L^\infty_\beta$ without any derivative,  and in \cite{DUYZ} the linearized Boltzmann equation with time-dependent force is studied. However, these are for the hard potential case, and to our best knowledge, in an early stage of the study, only Ukai-Asano \cite{UA} considered the moderately soft potential case $-1<\gamma\le 0$ on the whole space, where the semigroup method is used. The first and second authors in \cite{DSa} improved the result in \cite{UA} to include the full soft potential case $-3<\gamma< 0$, where they deduced decay rate in a Chemin-Lerner type space $\tilde{L}^\infty_T \tilde{L}^2_v (B^{3/2}_x)$. We remark that the slight modifications to  \cite{DSa} could be made to replace the Besov space $B^{3/2}_x$ by $L^\infty_x$ by using Guo's techniques as in \cite{Guo-soft,Guo-bd} and \cite{DHWY}. 

For the non-cutoff Boltzmann equation on the whole space, AMUXY~\cite{AMUXY-ARMA-2011} first proved the decay rate $(1+t)^{-1}$, but this was not optimal. Then Strain~\cite{Str12} proved the optimal decay of $L^2_v L^r_x$ norm with $r\ge 2$ for the soft potential case in $\mathbb{R}^3$. 
We note that, together with the $L^1_k \cap L^p_k$-interplay technique by \cite{KNN}, the splitting of frequency in the Fourier side by \cite{Str12} is also an important method we will employ.
Later \cite{SS} showed the $L^r_xL^2_v$-decay of the solution obtained by \cite{GS, Str12} for $2\le r\le \infty$ if initial data is (not only in the solution space) bounded in a negative Besov space $\dot{B}^{-q,\infty}_2L^2_v$, and in a similar spirit \cite{LZ-JDE-2021} proved the optimal decay of the solution for the cutoff case found in \cite{DLX-2016}.

In the end, we give a remark on the global existence of low-regularity solutions for the non-cutoff Boltzmann equation in the whole space. By low-regularity, it means that one does not expect to make use of the Sobolev embedding $H^2(\R^3_x)\subset L^\infty(\R^3_x)$ to obtain the $L^\infty$ bounds via the high-order derivative estimates. This is really hard for the angular non-cutoff potentials due to the velocity diffusion property of collision operators. We mention the only known results \cite{GHWO,KGH} for the Landau equation in torus and general bounded domains and \cite{AMSY20} for the non-cutoff Boltzmann equation in the torus. Here, \cite{AMSY20} used a De Giorgi argument for the weighted $L^2\cap L^\infty$ estimates on solutions with polynomial tails in large velocity as in \cite{GMM}. However, it is still a big problem to extend those results to the case when the space domain is unbounded for which the exponential decay structure of the equation is no longer valid. In this sense, the current work provides the first attempt to treat the low-regularity solutions for the whole space problem on the non-cutoff Boltzmann equation.

\section{Main results}\label{sec2}

Based on the $L^1_k$-$L^p_k$ interplay of  energy estimates, our goal is to show that, under similar assumptions as in Theorem \ref{DLSS Existence} except $\Omega=\mathbb{R}^3$, a low-regularity global solution to \eqref{BE} uniquely exists and it enjoys almost optimal polynomial time decay estimates.

In the following, $\mathbf{P}$ is an orthogonal projection from $L^2_v$ to $\ker L$, which takes the form
\begin{align*}
\mathbf{P}f(v)=[a+b\cdot v + c(\vert v\vert^2-3)]\mu^{1/2}(v),
\end{align*}
and $\vert \nabla_x \vert/\langle \nabla_x \rangle$ is a Fourier multiplier with the symbol $\vert k\vert/\sqrt{1+\vert k\vert^2}$.

Our main theorems are twofold. The first one is concerned with the global existence and large-time behavior of solutions for the hard potential case.

\begin{theorem}[Hard potentials]\label{thm: MainHard}
Assume \eqref{ass.ck} and \eqref{ass.add}. Let $\Omega=\mathbb{R}^3$, $\gamma+2s\ge 0$, $3/2<p\le \infty$ and 
\begin{align*}
\sigma=3\Big(1-\frac{1}{p}\Big) -2\varepsilon
\end{align*}
with $\varepsilon>0$ arbitrary small. 
There is $\epsilon_1>0$  such that if $F_0(x,v)=\mu+\mu^{\frac{1}{2}}f_0(x,v)\ge 0$ and
\begin{align*}
\Vert  f_0\Vert_{L^1_k L^2_v} +\Vert  f_0\Vert_{L^p_k L^2_v} \le \epsilon_1,
\end{align*}
then the Cauchy problem \eqref{BE} admits a unique global mild solution $f(t,x,v)$, $t>0$, $x\in \mathbb{R}^3$, $v\in \mathbb{R}^3$, which satisfies $F(t,x,v)=\mu+\mu^{\frac{1}{2}}f(t,x,v)\ge 0$ as well as the following estimates:
\begin{multline}
\Vert  (1+t)^{\frac{\sigma}{2}} f \Vert_{L^1_k L^\infty_T L^2_v} 
+\Vert   (1+t)^{\frac{\sigma}{2}} (\mathbf{I}-\mathbf{P}) f \Vert_{L^1_k L^2_T L^2_{v,D}}\\
+\Big \Vert (1+t)^{\frac{\sigma}{2}} \frac{\vert \nabla_x \vert}{\langle \nabla_x \rangle} (a,b,c) \Big\Vert_{L^1_k L^2_T}
\leq C \|f_0\|_{L^1_kL^2_v}+C  \Vert f_0 \Vert_{L^p_k L^2_v}, \label{thm.hc.c1}
\end{multline}
and 
\begin{equation}
\Vert f\Vert_{L^p_k L^\infty_T L^2_v} 
+\Vert (\mathbf{I}-\mathbf{P}) f\Vert_{L^p_k L^2_T L^2_{v,D}} 
+ \Big\Vert \frac{\vert \nabla_x \vert}{\langle \nabla_x \rangle } (a,b,c)\Big\Vert_{L^p_k L^2_T}
\le C \Vert f_0 \Vert_{L^p_k L^2_v}, 
\label{thm.hc.c2}
\end{equation}
for any $T>0$. Here $C$ is independent of $T$. 
\end{theorem}

The other main theorem corresponds to the soft potential case. The weight function in the statement is
\begin{align}\label{Weight Def}
w_{\ell,q}(v)=
\begin{cases}
\langle v \rangle^{\ell \vert \gamma+2s\vert},
&\ell>0,\ q=0\quad (\text{polynomial\ weight}),\\
e^{q\langle v\rangle/4}, 
&\ell=0,\ q>0\quad (\text{exponential\ weight}),\\
\langle v \rangle^{\ell \vert \gamma+2s\vert} e^{q\langle v\rangle/4},
&\ell>0,\ q>0\quad (\text{mixed\ weight}),
\end{cases}
\end{align}
and $(a)_+$ for $a\in \mathbb{R}$ stands for a value bigger than but arbitrary close to $a$.

\begin{theorem}[Soft potentials]\label{thm: MainSoft}
Assume \eqref{ass.ck} and \eqref{ass.add}. Let $\Omega=\mathbb{R}^3$, $\gamma+2s<0$, 
$3/2<p\le \infty$ and 
\begin{align*}
\sigma=3\Big(1-\frac{1}{p}\Big) -2\varepsilon
\end{align*}
with $\varepsilon>0$ arbitrary small. Also let $\ell\ge 0$, $q\ge 0$, $w_{\ell,q}$ be any of the weights in \eqref{Weight Def}, and suppose 
\begin{align}\label{ineq:SoftIndexCondition}
\Big(j>\frac{\sigma}{2r}\Big)
\wedge
\Big(q=0\Rightarrow \ell>\frac{\sigma}{2r}\Big),
\end{align} 
where
\begin{align*}
0<r<\frac{p'\varepsilon}{3+p'},\quad
\frac{1}{p}+\frac{1}{p'}=1.
\end{align*}
Then there is $\epsilon_2>0$ such that if $F_0(x,v)=\mu+\mu^{\frac{1}{2}}f_0(x,v)\ge 0$ and 
\begin{align}\label{MainSoftInitial}
\Vert w_{\ell+j,q} f_0 \Vert_{L^1_k L^2_v} 
+ \Vert w_{\ell+j,q} f_0\Vert_{L^p_k L^2_v} \le \epsilon_2,
\end{align}
then  the Cauchy problem \eqref{BE} admits a unique global mild solution $f(t,x,v)$, $x\in \mathbb{R}^3$, $v\in \mathbb{R}^3$, which satisfies $F(t,x,v)=\mu+\mu^{\frac{1}{2}}f(t,x,v)\ge 0$ as well as the following estimates: 
\begin{multline}
\Vert (1+t)^{\sigma/2} w_{\ell,q} f\Vert_{L^1_k L^\infty_T L^2_v}
+\Vert (1+t)^{\sigma/2} w_{\ell,q} (\mathbf{I}-\mathbf{P})  f \Vert_{L^1_k L^2_T L^2_{v,D}}\\
+\Big\Vert \frac{\vert \nabla_x \vert}{\langle \nabla_x \rangle} (1+t)^{\sigma/2} (a,b,c) \Big\Vert_{L^1_k L^2_T}
\leq   C\Vert w_{\ell+j,q}f_0\Vert_{L^1_k L^2_v}
+ C \Vert w_{\ell+j,q} f_0\Vert_{L^p_k L^2_v},
\label{thm.sc.c1}
\end{multline}
and 
\begin{multline}
\Vert w_{\ell,q} f\Vert_{L^p_k L^\infty_T L^2_v}
+\Vert  w_{\ell,q} (\mathbf{I}-\mathbf{P})  f \Vert_{L^p_k L^2_T L^2_{v,D}}\\
+\Big\Vert \frac{\vert \nabla_x \vert}{\langle \nabla_x \rangle}  (a,b,c) \Big\Vert_{L^p_k L^2_T}
\leq  C \Vert w_{\ell+j,q} f_0\Vert_{L^p_k L^2_v},
\label{thm.sc.c2}
\end{multline}
for any $T>0$. Here $C$ is independent of $T$.
\end{theorem}

To the best of our knowledge, Theorems \ref{thm: MainHard} and \ref{thm: MainSoft} provide the first result on the global existence of a class of low-regularity solutions for the non-cutoff Boltzmann equation in the whole space without relying on the Sobolev embedding $H^2(\R^3_x)\subset L^\infty(\R^3_x)$ via the high-order derivative estimates. In the meantime, we also obtain an explicit polynomial decay rate of solutions to the nonlinear Cauchy problem \eqref{BE}  by
\begin{equation}
\label{thm.re1}
\|f(t)\|_{L^1_kL^2_v} \leq C\|f_0\|_{(L^1_k\cap L^p_k)L^2_v} (1+t)^{-\frac{3}{2} (1-\frac{1}{p})+\varepsilon}   \quad \text{(hard potentials)}
\end{equation}
\begin{equation}
\label{thm.re2}
\|w_{\ell,q}f(t)\|_{L^1_kL^2_v} \leq C\|w_{\ell+j,q}f_0\|_{(L^1_k\cap L^p_k)L^2_v} (1+t)^{-\frac{3}{2} (1-\frac{1}{p})+\varepsilon}   \quad \text{(soft potentials)}
\end{equation}
for any $t\geq 0$, where $3/2<p\leq \infty$ is arbitrary and $0<\varepsilon<1-\frac{3}{2p}$ is arbitrarily chosen such that $\sigma=3(1-\frac{1}{p}) -2\varepsilon>1$.  

In what follows we give some remarks on such the result. First, since $\varepsilon$ can be sufficiently small, from \eqref{thm.re1} and \eqref{thm.re2} we have obtained an almost optimal decay rate $(1+t)^{-\frac{3}{2}(1-\frac{1}{p})_+}$ of solutions in the norm of $L^1_kL^2_v$ for initial data in the space $(L^1_k\cap L^p_k)L^2_v$ with small enough norm. Here, by an optimal rate, it means that the rate $(1+t)^{-\frac{3}{2}(1-\frac{1}{p})}$ is consistent with the one for either the heat equation or the linearized Boltzmann equation in the whole space in the framework of $L^\infty-L^{p'}$ time-decay in terms of the Hausdorff-Young inequality. Due to the technique of proofs, it is still a problem for us to remove a small gap with the parameter $\varepsilon$ in order to get an optimal rate. Notice that \eqref{thm.re1} and \eqref{thm.re2} can be verified at the linear level even in case without $\varepsilon$. However, it is highly nontrivial to prove them for the nonlinear problem due to low-regularity of solutions and initial data. 

Second, in the case of soft potentials, the velocity weight function is necessarily used to compensate for the degenerate dissipation of the linearized operator in large velocities. Notice that the rate is polynomial and independent of the weight function since it is essentially determined by the heat diffusion property over the low-frequency domain. Thus it seems most reasonable to make use of the pure polynomial weight as in the first case of \eqref{Weight Def} with $\ell$ satisfying \eqref{ineq:SoftIndexCondition}. 

Last, we should emphasize that the global existence is established under the assumption on the smallness of norms in both spaces $L^1_kL^2_v$ and $L^p_kL^2_v$ in contrast with the torus case \cite{DLSS} where only smallness in $L^1_kL^2_v$ was used. This means that the time-decay of solutions essentially induced by finite propagation of the norm in $L^p_kL^2_v$ plays a vital role in the proof of global existence, which is similar to \cite[Theorem 9.1 and Remark 9.1]{UA}.  It is unclear how to obtain the global existence without using the  $L^p_kL^2_v$ norm for the case of the whole space. Moreover, it is even unclear whether or not one can relax the smallness of the  $L^p_kL^2_v$ norm for initial data to the only boundedness which should look more reasonable in treating the nonlinear problem, for instance, see \cite[Theorem 2.4]{GW} or \cite[Theorems 1.1 and 1.3]{SS}. Those questions can be left to interested readers.

The key ideas for the proof of Theorems \ref{thm: MainHard} and \ref{thm: MainSoft} are, as we have briefly explained in the previous section, given as follows:
\begin{enumerate}
\item[(a)] As we can see from the statements, we cannot close the estimate by assuming that merely the $L^1_k$-norm of initial data is small without using the $L^p_k$-norm. This is because we have no control of the remainder term of time-weighted estimates which contains $(1+t)^{\sigma-1}$. The estimates of such terms are carried out via the interplay of $L^1_k$ and $L^p_k$ estimates. Finally, these are integrated with time-weighted estimates, as in \cite{KNN}.

\item[(b)] Non-time-weighted estimates will be proved under \eqref{apaweak}, the smallness assumption of $\Vert (1+t)^{\sigma/2}\mathbf{P}f\Vert_{L^1_k L^2_T L^2_{v,D}}\simeq \Vert (1+t)^{\sigma/2}(a,b,c)\Vert_{L^1_k L^2_T}$. As the main theorems indicate, we require different dissipation estimates for the microscopic and macroscopic parts. 
The assumption helps us to estimate the nonlinear terms by dividing each of them into four terms via
\begin{align*}
(\text{micro}+\text{macro})\times (\text{micro}+\text{macro}).
\end{align*}
The macro term coming from nonlinear terms does not have $\vert \nabla_x\vert$ over the low-frequency regime, which is necessary for the term to be absorbed into. Thus we need the assumption. This will be finally verified with the aid of time-weighted estimates. Notice that one may use the $L^6$-$L^3$ Young inequality in the $L^2_x$ framework as in \cite{Guo-IUMJ}, but such argument fails in the framework of $L^\infty_x$ that behaves similarly as $L^1_k$.

\item[(c)] For the soft potential case, along with non-weighted estimates for the hard potential case, it is known that velocity-weighted estimates are also needed. Especially, as we can see in the statement, we use different energy and dissipation terms for $w_{\ell,q} (\mathbf{I}-\mathbf{P})f \mathbf{1}_{\vert k\vert \le 1}$ and $w_{\ell,q}f \mathbf{1}_{\vert k\vert \ge 1}$. We owe the use of these terms to  \cite{Str12}.

\item[(d)] The condition $p>3/2$ comes from the restriction $\sigma>1$. Indeed, by definition of $\si$, it is equivalent to require $3(1-\frac{1}{p})>1$ since $\varepsilon$ can be arbitrarily small.   The index $3/2$ in the Fourier space $\R^3_k$ corresponds to $3$ in the phase space $\R^3_x$. This correspondence in the critical case of $p=3/2$ suggests a certain relation between the existence theories of the Boltzmann equation and the incompressible Navier-Stokes equations in the low-regularity existence framework, see the classical result Kato \cite{Kato} using the space $L^3(\R^3_x)$.
\end{enumerate}

We remark that, as in \cite{DLSS}, a similar statement will be true for the Landau equation due to the coincidence of the proof strategy, though we do not pursue this problem in this article.

Also, we can consider the general spatial domain $\mathbb{R}^d_x$ of arbitrary dimensions $d$. Considering the integrability of $(1+t)^{-\sigma}$ with $\sigma=d(1-1/p)-2\varepsilon$, we can extend the current results to the cases $d\ge 2$. The case $d=1$ is left open.

In the following, $C>0$ stands for a large constant that may vary line by line. Also, $\eta>0$ stands for a small constant which may also vary line by line, and $C_\eta>0$ stands for a large constant that diverges as $\eta\rightarrow 0$.

\section{A priori estimates for the hard potential case}

In this section, we will prove Theorem \ref{thm: MainHard} in the case of hard potentials. Notice that some lemmas as well as their immediate variants can be also used for the soft potential case that is to be studied in the next section. We consider hard and soft cases in separate sections for the sake of a straightforward and clear argument. In particular, the proof is much simpler for hard potentials than for soft potentials.

\subsection{$L^1_k$-$L^p_k$ trilinear a priori estimates}

In this subsection, we will show some estimates in the framework of $L^1_k\cap L^p_k$ without any time or velocity weight.

We first show a technical lemma related to the upper bound of the trilinear term in the Fourier variable.
As in \cite{DLSS}, we introduce the Fourier transform of the nonlinear term $\Ga(f,g)$. Let  
\begin{equation*}
\hat{\Gamma}(\hat{f},\hat{g})(k,v)=\int_{\mathbb{R}^3}\int_{\mathbb{S}^2} B(v-u,\sigma)\mu^{1/2}(u) [ \hat{f}(u')*_k \hat{g}(v')(k)- \hat{f}(u)*_k \hat{g}(v)(k)]d\sigma du,
\end{equation*}
where the convolution $\ast_k$ with respect to the frequency variable is defined as
\begin{equation*}
\hat{f}(u)*_k \hat{g}(v)(k)=\int_{\mathbb{R}^3} \hat{f}(k-\ell,u)\hat{g}(\ell,v)\, d\ell.
\end{equation*}

\begin{lemma}\label{ineq: Boltzmann nonlinear}
For $0<s<1$ and $\gamma>\max\{-3,-3/2-2s\}$, it holds
\begin{equation}\label{lem.ibn1}
\left| \left(\hat{\Gamma}(\hat{f},\hat{g})(k),\hat{h}(k)\right)_{L^2_v}\right|\le C\int_{\mathbb{R}^3_\ell}\Vert \hat{f}(k-\ell)\Vert_{L^2_v}\Vert \hat{g}(\ell)\Vert_{L^2_{v,D}}  \Vert \hat{h}(k)\Vert_{L^2_{v,D}} d\ell.
\end{equation}
\end{lemma}

\begin{proof}
By definition and Fubini's theorem we have
\begin{align*}
&\vert (\hat{\Gamma}(\hat{f},\hat{g})(k),\hat{h}(k))_{L^2_v}\vert\\
&=\vert \int_{\mathbb{R}^d_v} \int_{\mathbb{R}^d_{v_*}} \int_{\mathbb{S}^{d-1}} B (\mu\mu_*)^{1/2} (\hat{f'_*}*\hat{g'}(k)-\hat{f_*}*\hat{g}(k))d\sigma dv_*\bar{\hat{h}}(k)dv\vert\\
&=\vert \int_{\mathbb{R}^d_v} \int_{\mathbb{R}^d_{v_*}} \int_{\mathbb{S}^{d-1}} B (\mu\mu_*)^{1/2}\int_{\mathbb{R}^d_\ell} (\hat{f'_*}(k-\ell)\hat{g'}(\ell)-\hat{f_*}(k-\ell)\hat{g}(\ell))d\ell d\sigma dv_* \bar{\hat{h}}(k)dv\vert\\
&=\vert \int_{\mathbb{R}^d_\ell} \int_{\mathbb{R}^d_v} \Gamma(\hat{f}(k-\ell),\hat{g}(\ell))\bar{\hat{h}}(k)dv  d\ell \vert\\
&\le \int_{\mathbb{R}^d_\ell} \vert (\Gamma(\hat{f}(k-\ell),\hat{g}(\ell)),\hat{h})_{L^2_v}\vert d \ell.
\end{align*}
Now we recall \cite[Theorem 1.2]{AMUXY-KRM-2013} that if $0<s<1$  and $\gamma>\{-3,-3/2-2s\}$, it holds
\begin{align}\label{ineq:Trilinear}
\vert (\Gamma(f,g),h)_{L^2_v}\vert \le C \Vert f\Vert_{L^2_v} \Vert g \Vert_{L^2_{v,D}} \Vert h \Vert_{L^2_{v,D}}.
\end{align}
With this inequality, the desired estimate \eqref{lem.ibn1} is readily obtained. The proof of Lemma \ref{ineq: Boltzmann nonlinear} is complete.
\end{proof}

Consider the Cauchy problem \eqref{BE}. By the energy estimates in $L^1_k\cap L^p_k$ with the help of Lemma \ref{ineq: Boltzmann nonlinear}, we obtain

\begin{lemma}\label{lem: micro a priori}
Let $p>1$.
There exists $C>0$ such that the solution to \eqref{BE} satisfies
\begin{equation}
\Vert f\Vert_{L^1_k L^\infty_T L^2_v}+\Vert (\mathbf{I}-\mathbf{P})f \Vert_{L^1_k L^2_T L^2_{v,D}}
\le C \Vert f_0 \Vert_{L^1_k L^2_v} +C\Vert f\Vert_{L^1_k L^\infty_T L^2_v} \Vert f\Vert_{L^1_k L^2_T L^2_{v,D}}, \label{ineq: L^1_k micro a priori}
\end{equation}
and 
\begin{equation}
\Vert f\Vert_{L^p_k L^\infty_T L^2_v}+\Vert (\mathbf{I}-\mathbf{P})f \Vert_{L^p_k L^2_T L^2_{v,D}}
\le C \Vert f_0 \Vert_{L^p_k L^2_v} +C\Vert f\Vert_{L^p_k L^\infty_T L^2_v} \Vert f\Vert_{L^1_k L^2_T L^2_{v,D}}, \label{ineq: L^infty_k micro a priori}
\end{equation}
for any $T>0$.
\end{lemma}

\begin{proof}
First, we prove \eqref{ineq: L^1_k micro a priori} in the same way as in \cite{DLSS} and \eqref{ineq: L^infty_k micro a priori} will follow by a slight modification.
Applying the Fourier transform to the equation with respect to $x$, we have
\begin{align}\label{eq: fourier}
\partial_t \hat{f}+iv\cdot k \hat{f}+L\hat{f}=\hat{\Gamma}(\hat{f},\hat{f}).
\end{align}
We multiply \eqref{eq: fourier} with $\bar{\hat{f}}(t,k,v)$ and take the real part to get
\begin{align*}
\frac{1}{2}\frac{d}{dt} \vert \hat{f}(t,k,v)\vert^2 +\mathrm{Re} (L\hat{f}, \hat{f})=\mathrm{Re} (\hat{\Gamma}(\hat{f},\hat{f}),\hat{f}),
\end{align*}
where $(\cdot,\cdot)$ denotes the usual complex inner product.
Integrating this over $[0,T]\times \mathbb{R}^3_v$ gives
\begin{align}\label{eq: dt}
\frac{1}{2} \Vert \hat{f}(t,k)\Vert_{L^2_v}^2 +\int^T_0 \mathrm{Re} (L\hat{f},\hat{f})_{L^2_v} d\tau =\frac{1}{2} \Vert \hat{f}_0 (k)\Vert_{L^2_v}^2+\int^T_0  \mathrm{Re} (\hat{\Gamma}(\hat{f},\hat{f}),\hat{f})_{L^2_v} d\tau.
\end{align}
Recall that the linearized  collision operator $L$ satisfies the coercivity estimate~\cite[Proposition 2.1]{AMUXY-2012-JFA}
\begin{align*}
 (Lg,g)_{L^2_v}\geq  \delta_0 \Vert (\mathbf{I}-\mathbf{P})g \Vert_{L^2_{v,D}}^2
\end{align*}
for some $\delta_0>0$.
Substituting this estimate into \eqref{eq: dt} and taking the $L^\infty_T$ norm yield
\begin{align}
\Vert \hat{f}(k)\Vert_{L^\infty_T L^2_v} + \Big(\int^T_0 \Vert (\mathbf{I}-\mathbf{P})\hat{f}(t,k)\Vert_{L^2_{v,D}}^2 dt \Big)^{1/2}\notag \\
\le C \Vert \hat{f}_0(k)\Vert_{L^2_v}+C\Big(\int^T_0 \vert (\hat{\Gamma}(\hat{f},\hat{f}), \hat{f})_{L^2_v}\vert^2 dt\Big)^{1/2}. \label{lem.nont.p1}
\end{align}
Noticing $\Gamma(f,f)\in (\ker L)^\perp$, we can replace the inner product $( \hat{\Gamma}(\hat{f},\hat{f}), \hat{f})_{L^2_v}$ with $(\hat{\Gamma}(\hat{f},\hat{f}), (\mathbf{I}-\mathbf{P})\hat{f})_{L^2_v}$. Owing to Lemma \ref{ineq: Boltzmann nonlinear}, the last term of \eqref{lem.nont.p1} is bounded as
\begin{align*}
&\Big(\int^T_0 \vert (\hat{\Gamma}(\hat{f},\hat{f}), \hat{f})_{L^2_v}\vert^2 dt\Big)^{1/2}\notag \\
&\leq C\Big( \int^T_0 \int_{\mathbb{R}^3} \Vert \hat{f}(k-\ell)\Vert_{L^2_v} \Vert \hat{f}(\ell)\Vert_{L^2_{v,D}} \Vert (\mathbf{I}-\mathbf{P})\hat{f}(k)\Vert_{L^2_{v,D}} d\ell dt \Big)^{1/2}\\
&\le C\Big( \int^T_0 \Big( \int_{\mathbb{R}^3} \Vert \hat{f}(k-\ell) \Vert_{L^2_v} \Vert \hat{f}(\ell)\Vert_{L^2_{v,D}} d\ell\Big)^2dt \Big)^{1/4} \Big( \int^T_0 \Vert (\mathbf{I}-\mathbf{P}) \hat{f}(k)\Vert_{L^2_{v,D}}^2 dt\Big)^{1/4}\\
&\le \eta \Big(\int^T_0 \Vert (\mathbf{I}-\mathbf{P}) \hat{f}(k)\Vert_{L^2_{v,D}}^2 dt \Big)^{1/2}\\
&\quad+\frac{C}{4\eta} \Big( \int^T_0 \Big( \int_{\mathbb{R}^3} \Vert \hat{f}(k-\ell) \Vert_{L^2_v} \Vert \hat{f}(\ell)\Vert_{L^2_{v,D}} d\ell\Big)^2dt \Big)^{1/2},
\end{align*}
where $\eta>0$ is an arbitrary constant. 
Using the Minkowski  inequality, the integral term with the coefficient $C/(4\eta)$ above is further bounded as
\begin{align*}
&\Big( \int^T_0 \Big( \int_{\mathbb{R}^3} \Vert \hat{f}(k-\ell) \Vert_{L^2_v} \Vert \hat{f}(\ell)\Vert_{L^2_{v,D}} d\ell\Big)^2dt \Big)^{1/2}\\
&\leq \int_{\mathbb{R}^3} \Big(\int^T_0 \Vert \hat{f}(k-\ell) \Vert_{L^2_v}^2 \Vert \hat{f}(\ell)\Vert_{L^2_{v,D}}^2 dt \Big)^{1/2} d\ell\\
&\le \int_{\R^3} \|\hat{f}(k-\ell)\|_{L^\infty_TL^2_v}\|\hat{f}(\ell)\|_{L^2_TL^2_{v,D}}\,d\ell.
\end{align*}
Combining the above two estimates, plugging it back to \eqref{lem.nont.p1}, taking the $L^1_k$ norm and using
\begin{equation}\notag
\int_{\R^3_k} \|\hat{f}\|_{L^\infty_TL^2_v}\ast_k\|\hat{f}\|_{L^2_TL^2_{v,D}}(k)\,dk\leq \Vert f\Vert_{L^1_k L^\infty_T L^2_v} \Vert f\Vert_{L^1_k L^2_T L^2_{v,D}},
\end{equation}
we then arrive at
\begin{align*}
\Vert f\Vert_{L^1_k L^\infty_T L^2_v}+\Vert (\mathbf{I}-\mathbf{P})f \Vert_{L^1_k L^2_T L^2_{v,D}}
&\le C\Vert f_0 \Vert_{L^1_k L^2_v} + C\eta \Vert (\mathbf{I}-\mathbf{P})f \Vert_{L^1_k L^2_T L^2_{v,D}} \\
&\quad+\frac{C}{\eta}\Vert f\Vert_{L^1_k L^\infty_T L^2_v} \Vert f\Vert_{L^1_k L^2_T L^2_{v,D}}.
\end{align*}
Letting $\eta>0$ be small enough, we obtain the desired estimate \eqref{ineq: L^1_k micro a priori} from above. 
Notice that \eqref{ineq: L^infty_k micro a priori} can be also proved in the same way just by taking the $L^p_k$ norm instead of $L^1_k$ and applying
\begin{equation}\notag
\left\|\|\hat{f}\|_{L^\infty_TL^2_v}\ast_k\|\hat{f}\|_{L^2_TL^2_{v,D}}\right\|_{L^p(\R^3_k)}\leq \Vert f\Vert_{L^p_k L^\infty_T L^2_v} \Vert f\Vert_{L^1_k L^2_T L^2_{v,D}}.
\end{equation}
The proof of Lemma \ref{lem: micro a priori} is complete.
\end{proof}

For further treating the nonlinear terms in Lemma \ref{lem: micro a priori}, we can obtain the closed estimates under a certain smallness assumption. To the end, for any $T>0$ we set 
\begin{equation}
N(T)=\|f\|_{L^1_kL^\infty_TL^2_v} +\int_{\R^3} \sup_{0<t<T}(1+t)^{\frac{\si}{2}}|\widehat{(a,b,c)}(t,k)|\,dk\notag
\end{equation}
with $\sigma=3(1-1/p)-2\varepsilon$ for a small enough constant $\varepsilon>0$. We make the a priori assumption
\begin{equation}
\label{apaweak}
N(T)\leq \delta
\end{equation}
with $\delta>0$ small enough. 

\begin{lemma}\label{lem: micro est hard}
Let $3/2<p\le \infty$.
Assume that $\|f_0\|_{L^1_kL^2_v}$ is sufficiently small. Under the assumption \eqref{apaweak}, 
one has
\begin{align}
\label{p1kc}
 \|f\|_{L^1_kL^\infty_TL^2_v}+\|(\mathbf{I}-\mathbf{P})f\|_{L^1_kL^2_TL^2_{v,D}}
 &\leq C \|f_0\|_{L^1_kL^2_v},
\end{align}
and 
\begin{align}
  \|f\|_{L^p_kL^\infty_TL^2_v}+\|(\mathbf{I}-\mathbf{P})f\|_{L^p_kL^2_TL^2_{v,D}}
 &\leq C \|f_0\|_{L^p_kL^2_v},
  \label{pinftykc}
\end{align}
for any $T>0$.
\end{lemma}

\begin{proof}
We first show \eqref{p1kc}.
Decomposing $f=\mathbf{P}f+(\mathbf{I}-\mathbf{P})f$, we write
\begin{align*}
\Ga(f,f)=\Ga(f,\mathbf{P}f)+\Ga(f,(\mathbf{I}-\mathbf{P})f).
\end{align*}
Notice that the a priori assumption 
\eqref{apaweak} tells
\begin{align*}
\|\mathbf{P}f\|_{L^1_kL^2_TL^2_{v,D}}\leq C \|(a,b,c)\|_{L^1_kL^2_T}\leq C\delta\left(\int_0^T (1+t)^{-\si}dt\right)^{\frac{1}{2}}\leq C\delta
\end{align*}
due to $\si=3(1-1/p)-2\varepsilon>1$.  With this as well as \eqref{apaweak} again, it follows 
\begin{align*}
\Vert f\Vert_{L^1_k L^\infty_T L^2_v} \Vert f\Vert_{L^1_k L^2_T L^2_{v,D}} &\leq \Vert f\Vert_{L^1_k L^\infty_T L^2_v} (\Vert \mathbf{P} f\Vert_{L^1_k L^2_T L^2_{v,D}}+\Vert (\mathbf{I}-\mathbf{P})f\Vert_{L^1_k L^2_T L^2_{v,D}})\\
&\leq C \delta (\|f\|_{L^1_kL^\infty_TL^2_v}+\| (\mathbf{I}-\mathbf{P}) f\|_{L^1_kL^2_TL^2_{v,D}}).
\end{align*}
Together with \eqref{ineq: L^1_k micro a priori}, this leads to \eqref{p1kc} as $\delta>0$ can be small enough. 
Once \eqref{p1kc} is proved, it is straightforward to deduce \eqref{pinftykc} from \eqref{ineq: L^infty_k micro a priori} because
\begin{align*}
\Vert f\Vert_{L^p_k L^\infty_T L^2_v} \Vert f\Vert_{L^1_k L^2_T L^2_{v,D}}\le C (\Vert f_0\Vert_{L^1_kL^2_v}+\delta)\Vert f\Vert_{L^p_k L^\infty_T L^2_v},
\end{align*}
where $\Vert f_0\Vert_{L^1_kL^2_v}$ can also be small enough. 
This completes the proof of Lemma \ref{lem: micro est hard}. 
\end{proof}

It remains to close the a priori assumption \eqref{apaweak}.

\subsection{Macroscopic estimates}

Next, we derive an a priori estimate for the macroscopic part $\mathbf{P}f\sim[a,b,c]$. Before showing that, we shall prove two lemmas which will be used later. 
\begin{lemma}\label{lem: inner product Gamma}
Let $1\le p\le \infty$.
For any $\phi \in \mathcal{S}(\mathbb{R}^3_v)$ there exists $C_\phi>0$ such that it holds
\begin{align*}
\Big\Vert  \Big( \int^T_0 \vert (\hat{\Gamma}(\hat{f},\hat{g}),\phi )_{L^2_v}\vert^2 dt \Big)^{1/2} \Big\Vert_{L^p_k}  \le C_\phi\Vert f \Vert_{L^p_k L^\infty_T L^2_v}  \Vert g\Vert_{L^1_k L^2_T L^2_{v,D}},
\end{align*}
for any $T>0$.
\end{lemma}

\begin{proof}
We consider only the case $p=\infty$. Thanks to $\Vert fg\Vert_{L^p_k}\le \Vert f\Vert_{L^p_k}\Vert g\Vert_{L^1_k}$, the other cases can be similarly handled.
By Lemma \ref{ineq: Boltzmann nonlinear} and Minkowski's inequality, we have
\begin{align*}
&\sup_k  \Big( \int^T_0 \vert (\hat{\Gamma}(\hat{f},\hat{g}),\phi )_{L^2_v}\vert^2 dt \Big)^{1/2} \\
&\le \sup_k \Big( \int^T_0 \Big(\int_{\mathbb{R}^3_\ell}\Vert \hat{f}(k-\ell)\Vert_{L^2_v}  \Vert \hat{g}(\ell)\Vert_{L^2_{v,D}} \Vert \phi \Vert_{L^2_{v,D}}  d\ell \Big)^2d t \Big)^{1/2} \\
&\le C_\phi \sup_k \int_{\mathbb{R}^3_\ell} \Big(\int^T_0 \Vert \hat{f}(k-\ell)\Vert_{L^2_v}^2 \Vert \hat{g}(\ell) \Vert_{L^2_{v,D}}^2 dt \Big)^{1/2} d\ell \\
&\le C_\phi \sup_k \sup_{0\le t\le T} \Vert \hat{f}(k)\Vert_{L^2_v}  \cdot \int_{\mathbb{R}^3_\ell} \Big(\int^T_0 \Vert  \hat{g}(\ell) \Vert_{L^2_{v,D}}^2 dt \Big)^{1/2} d\ell.
\end{align*}
This is the desired estimate, and the proof of Lemma \ref{lem: inner product Gamma} is complete.
\end{proof}

\begin{lemma}\label{lem: inner product L}
Let $1\le p\le \infty$. Under the same assumption as in the previous lemma, there exists $C'_\phi>0$ such that it holds
\begin{align*}
\Big\Vert \Big( \int^T_0 \vert (L\hat{f},\phi )_{L^2_v}\vert^2 dt \Big)^{1/2} \Big\Vert_{L^p_k}  \le C'_\phi \Vert f\Vert_{L^p_k L^2_T L^2_{v,D}},
\end{align*}
for any $T>0$.
\end{lemma}

\begin{proof}
We only consider the case $p=\infty$ again.
Notice $L\hat{f}=\Gamma(\hat{f},\mu^{1/2})+\Gamma(\mu^{1/2},\hat{f})$. Thus the same calculation in the previous proof works. The estimate of the second term is readily obtained. The first term is bounded by
\begin{align*}
\sup_k \Big( \int^T_0 \Vert \hat{f}(k)\Vert_{L^2_v}^2 \Vert \mu^{1/2} \Vert_{L^2_{v,D}}^2 \Vert \phi \Vert_{L^2_{v,D}}^2 dt\Big)^{1/2},
\end{align*}
so we recall the estimate $\Vert \cdot \Vert_{L^2_v}\le \Vert \cdot \Vert_{L^2_{(\gamma/2+s)}} \le C \Vert \cdot \Vert_{L^2_{v,D}}$ for the hard potential case $\gamma+2s\ge 0$. For the soft potential case, we cite~\cite[Lemma 2.15]{AMUXY-2012-JFA}, which shows $\vert (\Gamma(f,\mu^{1/2})),\phi)_{L^2_v}\vert \le C \Vert \mu^{10^{-3}} f\Vert_{L^2_v}$. Since the last norm is bounded by $\Vert f \Vert_{L^2_{v,D}}$, we complete the proof of Lemma \ref{lem: inner product L}.
\end{proof}

Now, considering the macro-micro decomposition of the equation
\begin{align}\label{eq: linearized BE}
\begin{cases}
\partial_t f +v\cdot \nabla_x f +Lf=H,\\
f(0,x,v)=f_0(x,v),
\end{cases}
\end{align}
we will have the a priori estimates on the macro part in the following lemma. Proof of this lemma will be put in Appendix since it is a standard procedure.

\begin{lemma}\label{lem: macro a priori}
Let $1\le p< \infty$.
Let $f$ be a solution to \eqref{eq: linearized BE} with an inhomogeneous term $H=H(t,x,v)$ such that $\mathbf{P}H\equiv 0$. Then it holds that
\begin{align*}
\Big\Vert \frac{\vert \nabla_x \vert}{\langle \nabla_x \rangle } (a,b,c) \Big\Vert_{L^p_k L^2_T} 
&\le C \Vert f_0 \Vert_{L^p_k L^2_v} + \Vert f\Vert_{L^p_k L^\infty_T L^2_v} + \Vert (\mathbf{I}-\mathbf{P})f\Vert_{L^p_k L^2_T L^2_{v,D}} \\
&\quad+\Big(\int_{\mathbb{R}^3_k} \Big(\int^T_0 \frac{1}{1+\vert k\vert^2}\vert (\hat{H}, \mu^{1/4})_{L^2_v}\vert^2 d\tau \Big)^{p/2} dk\Big)^{1/p}, 
\end{align*}
and 
\begin{align*}
\Big\Vert \frac{\vert \nabla_x \vert}{\langle \nabla_x \rangle } (a,b,c) \Big\Vert_{L^\infty_k L^2_T} 
&\le C \Vert f_0 \Vert_{L^\infty_k L^2_v} + \Vert f\Vert_{L^\infty_k L^\infty_T L^2_v} + \Vert (\mathbf{I}-\mathbf{P})f\Vert_{L^\infty_k L^2_T L^2_{v,D}} \\
&\quad+\sup_{k\in\mathbb{R}^3} \Big(\int^T_0 \frac{1}{1+\vert k\vert^2}\vert (\hat{H}, \mu^{1/4})_{L^2_v}\vert^2 d\tau \Big)^{1/2},
\end{align*}
for any $T>0$.
\end{lemma}

With $H=\Gamma(f,f)$ in Lemma \ref{lem: macro a priori}, applying Lemma \ref{lem: inner product Gamma}, we have 
\begin{align*}
\Big\Vert \frac{\vert \nabla_x \vert}{\langle \nabla_x \rangle } (a,b,c) \Big\Vert_{L^p_k L^2_T} 
&\le C \Vert f_0 \Vert_{L^p_k L^2_v} + \Vert f\Vert_{L^p_k L^\infty_T L^2_v} + \Vert (\mathbf{I}-\mathbf{P}) f \Vert_{L^p_k L^2_T L^2_{v,D}} \\
&+C \Vert f\Vert_{L^p_k L^\infty_T L^2_v}\Vert f\Vert_{L^1_k L^2_T L^2_{v,D}},
\end{align*}
for $1\le p\le \infty$.
Therefore combining 
Lemma \ref{lem: micro est hard} and Lemma \ref{lem: macro a priori} readily gives

\begin{lemma}
Assume that $3/2<p\le \infty$ and $\|f_0\|_{L^1_kL^2_v}$ is sufficiently small. Under the assumption \eqref{apaweak}, it holds 
\begin{align}
\Vert  f \Vert_{L^1_kL^\infty_TL^2_v}
+\Vert (\mathbf{I}-\mathbf{P}) f\Vert_{L^1_kL^2_TL^2_{v,D}}
+\Big\Vert \frac{\vert \nabla_x \vert}{\langle \nabla_x \rangle } (a,b,c)\Big\Vert_{L^1_k L^2_T} 
\leq C \|f_0\|_{L^1_kL^2_v},
\label{p1kcc}
\end{align}
and 
\begin{align}
\Vert f\Vert_{L^p_k L^\infty_T L^2_v} 
+\Vert (\mathbf{I}-\mathbf{P}) f\Vert_{L^p_k L^2_T L^2_{v,D}} 
+ \Big\Vert \frac{\vert \nabla_x \vert}{\langle \nabla_x \rangle } (a,b,c)\Big\Vert_{L^p_k L^2_T}
\le C \Vert f_0 \Vert_{L^p_k L^2_v},
\label{pinfk}
\end{align}
for any $T>0$.
\end{lemma}

\subsection{Time-weighted estimates}
In order to conclude all the estimates, we have to close the a priori assumption \eqref{apaweak} when both $\Vert f_0\Vert_{L^1_k L^2_v}$ and  $\Vert f_0\Vert_{L^p_k L^2_v}$ are small enough.
Hence, we need to further consider the time-weighted estimates. First of all, we will show a consequence of time-weighted microscopic a priori estimates.

\begin{lemma}\label{lem: micro time-weighted}
Let $3/2<p\leq \infty$ and $\sigma=3(1-1/p)-2\varepsilon$. For any $T>0$,
it holds that
\begin{align}
&\Vert  (1+t)^{\sigma/2} f\Vert_{L^1_k L^\infty_T L^2_v}  
+\Vert (1+t)^{\sigma/2} (\mathbf{I}-\mathbf{P})f \Vert_{L^1_k L^2_T L^2_{v,D}} \notag \\
&\le C\eta \Big\Vert (1+t)^{\sigma/2} \frac{\vert \nabla_x \vert}{\langle \nabla_x \rangle } (a,b,c) \Big\Vert_{L^1_k L^2_T}+C_\eta (\Vert f_0\Vert_{L^1_k L^2_v} +\Vert f_0\Vert_{L^p_k L^2_v}),
\label{ineq: micro time-weighted}
\end{align}
where $\eta>0$ is an arbitrarily small constant.
\end{lemma}

\begin{proof}
In fact, repeating the proof of Lemma \ref{lem: micro a priori}, one can similarly deduce
\begin{align*}
&\Vert  (1+t)^{\sigma/2} f\Vert_{L^1_k L^\infty_T L^2_v}  
+\Vert (1+t)^{\sigma/2} (\mathbf{I}-\mathbf{P})f \Vert_{L^1_k L^2_T L^2_{v,D}}\\
&\le C\Vert f_0\Vert_{L^1_k L^2_v} 
+C\int_{\mathbb{R}^3}\Big( \int^T_0 (1+t)^\sigma \vert (\widehat{\Gamma(f,f)}, (\mathbf{I}-\mathbf{P})\hat{f} )_{L^2_v} \vert dt\Big)^{1/2}  dk\\
&\quad+C\sqrt{\sigma} \int_{\mathbb{R}^3}\Big( \int^T_0 (1+t)^{\sigma -1} \Vert \hat{f}\Vert_{L^2_v}^2 dt\Big)^{1/2} dk.
\end{align*}
As in the proof of Lemma \ref{lem: micro a priori}, we have
\begin{align*}
&\int_{\mathbb{R}^3}\Big( \int^T_0 (1+t)^\sigma \vert (\widehat{\Gamma(f,f)}, (\mathbf{I}-\mathbf{P})\hat{f} )_{L^2_v} \vert dt\Big)^{1/2}  dk \\
&\le \eta \Vert (1+t)^{\sigma/2} (\mathbf{I}-\mathbf{P})f \Vert_{L^1_k L^2_T L^2_{v,D}} + C_\eta \Vert (1+t)^{\sigma/2} f\Vert_{L^1_k L^\infty_T L^2_v} \Vert f\Vert_{L^1_k L^2_T L^2_{v,D}}\\
&\le \eta \Vert (1+t)^{\sigma/2} (\mathbf{I}-\mathbf{P})f \Vert_{L^1_k L^2_T L^2_{v,D}} +C_\eta \Vert f_0\Vert_{L^1_k L^2_v}\Vert (1+t)^{\sigma/2} f\Vert_{L^1_k L^\infty_T L^2_v},
\end{align*}
where we have used \eqref{p1kcc} in the last inequality.
Both terms are absorbed if $\eta$ and $\Vert f_0\Vert_{L^1_k L^2_v}$ are sufficiently small.
Thus we are left to deduce a bound of
\begin{align}
&\int_{\mathbb{R}^3}\Big( \int^T_0 (1+t)^{\sigma -1} \Vert \hat{f}\Vert_{L^2_v}^2 dt\Big)^{1/2} dk \notag \\
&=\Big(\int_{\vert k\vert \le 1}+\int_{\vert k\vert \ge 1}\Big)\Big( \int^T_0 (1+t)^{\sigma -1} \Vert \hat{f}\Vert_{L^2_v}^2 dt\Big)^{1/2} dk.
\label{pdt3}
\end{align}

We first consider the integral term over the high-frequency part $|k|\geq 1$ in \eqref{pdt3}. For $|k|\geq 1$, one sees
\begin{align*}
\frac{|k|^2}{1+|k|^2}\sim 1,
\end{align*}
so that for $\eta>0$ arbitrarily small,
\begin{align}
 (1+t)^{\si-1}\leq \{\eta (1+t)^\si +C_\eta\}\frac{|k|^2}{1+|k|^2}. \notag
\end{align}
For the term with $\eta>0$, \eqref{pdt3} can be bounded as 
\begin{align}
&\eta \Big\|(1+t)^{\si/2} \frac{\vert \nabla_x \vert}{\langle \nabla_x \rangle } f\Big\|_{L^1_kL^2_TL^2_v}\notag\\&\leq \eta \Vert (1+t)^{\sigma/2} \frac{\vert \nabla_x \vert}{\langle \nabla_x \rangle }(\mathbf{I}-\mathbf{P})f \Vert_{L^1_k L^2_T L^2_{v}}+\eta \Big\Vert (1+t)^{\sigma/2} \frac{\vert \nabla_x \vert}{\langle \nabla_x \rangle }\mathbf{P} f \Big\Vert_{L^1_k L^2_T L^2_{v}}\notag\\
&\leq \eta  \Vert (1+t)^{\sigma/2}(\mathbf{I}-\mathbf{P})f \Vert_{L^1_k L^2_T L^2_{v,D}}+C\eta \Big\Vert (1+t)^{\sigma/2} \frac{\vert \nabla_x \vert}{\langle \nabla_x \rangle } (a,b,c) \Big\Vert_{L^1_k L^2_T},
\label{lem.mtw.pa1}
\end{align}
where the first term can be absorbed and the second term with a small coefficient $C\eta$ remains to be estimated in the next lemma. Similarly, for the term with $C_\eta$,  \eqref{pdt3} can be bounded as  
\begin{multline}\notag
C_\eta \Big\|\frac{\vert \nabla_x \vert}{\langle \nabla_x \rangle}f \Big\|_{L^1_kL^2_TL^2_v}
\leq C_\eta  \Big(\|(\mathbf{I}-\mathbf{P})f\|_{L^1_kL^2_TL^2_{v,D}}+\Big\|\frac{\vert \nabla_x \vert}{\langle \nabla_x \rangle}(a,b,c)\Big\|_{L^1_kL^2_T}\Big)\\
\leq C_\eta  \|f_0\|_{L^1_kL^2_v},
\end{multline}
where we have used \eqref{p1kcc} in the last inequality.

Now, we move to the estimate on the integral term over the low-frequency part $|k|\leq 1$ in \eqref{pdt3}. We look for the interpolation
\begin{align}\label{eq:InterpolationHard}
\si-1=[\si,\si-\omega]_\theta:=(1-\theta)\si +\theta (\si-\omega),\quad
\omega:=4-\frac{3}{p}-\varepsilon>1.
\end{align}
Then one has to take $\theta=1/\omega\in (0,1)$.
Applying the Young inequality with  $\eta>0$ arbitrarily small,
\begin{align*}
(1+t)^{\si-1}=(1+t)^{(1-\theta)\si}(|k|^2)^{1-\theta}\cdot (1+t)^{\theta (\sigma-\omega)} (|k|^2)^{-(1-\theta)}\\
\leq \eta (1+t)^\si |k|^2 +C_\eta (1+t)^{\si -\omega} (|k|^2)^{-\frac{1-\theta}{\theta}}.
\end{align*}
Once again, for the term with $\eta>0$ suitably small, thanks to $|k|^2\leq \frac{2|k|^2}{1+|k|^2}$ over $|k|\leq 1$, \eqref{pdt3} can be bounded in completely the same way as \eqref{lem.mtw.pa1}.  For the term with $C_\eta$, using the idea in \cite{KNN}, we can bound \eqref{pdt3}  by
\begin{align}\label{ineq:InterpolationHard}
C_\eta \Vert f\Vert_{L^p_kL^\infty_T L^2_v} \Big(\int_{|k|\leq 1 }  |k|^{-p'\frac{1-\theta}{\theta}}dk\Big)^{1/p'}\, \left(\int_0^T (1+t)^{\sigma-\omega} dt\right)^{\frac{1}{2}},
\end{align}
where we have used the H\"older inequality  for $1/p+1/p'=1$ with the standard convention $\infty'=1$.
Since
\begin{align*}
\sigma-\omega=-1-\varepsilon,\quad
p'\frac{1-\theta}{\theta}=p'\Big(3-\frac{3}{p}-\varepsilon\Big)
=3-p'\varepsilon<3,
\end{align*}
it follows that \eqref{ineq:InterpolationHard} is further bounded, using \eqref{pinfk}, by
$C_{\eta}  \Vert f_0 \Vert_{L^p_k L^2_v}$. 
This then completes the proof.
\end{proof}

Next, we need to further consider time-weighted estimates of the macroscopic part as they appear on the right-hand side of \eqref{ineq: micro time-weighted} with a coefficient that can be arbitrarily small. 

\begin{lemma}\label{lem.twemp}
Let $3/2<p\le \infty$ and $f$ be a solution to \eqref{eq: linearized BE}. Then, it holds
\begin{align}
&\Big \Vert (1+t)^{\frac{\sigma}{2}} \frac{\vert \nabla_x \vert}{\langle \nabla_x \rangle} (a,b,c) \Big\Vert_{L^1_k L^2_T}=\int_{\mathbb{R}^3} \frac{\vert k\vert}{\sqrt{1+\vert k\vert^2}} \Big( \int^T_0 (1+t)^\sigma \vert (\hat{a}, \hat{b}, \hat{c})\vert^2 dt\Big)^{1/2} dk \notag \\
&\le C\Vert f_0 \Vert_{L^1_k L^2_v}
+C\Vert (1+t)^{\sigma/2} f\Vert_{L^1_k L^\infty_T L^2_v}
+C\Vert (1+t)^{\sigma/2} (\mathbf{I}-\mathbf{P})f \Vert_{L^1_k L^2_T L^2_{v,D}} \notag \\
&+C\Vert (a,b,c)\Vert_{L^p_k L^\infty_T}+C\int_{\mathbb{R}^3} \Big( \int^T_0 (1+t)^\sigma \frac{\vert (\hat{H}, \mu^{1/4})\vert^2}{1+\vert k\vert^2} dt \Big)^{1/2} dk,
\label{ineq: macro time-weighted}
\end{align}
for any $T>0$.
\end{lemma}

\begin{proof}
It suffices to make slight modifications to the proof of Lemma \ref{lem: macro a priori} by additionally considering the time-weighted estimates. For instance, as for obtaining the dissipation of $c$ in \eqref{ap.adp1}, we need to compute
\begin{align*}
(1+t)^\sigma \partial_t \Big(\Lambda_j ((\mathbf{I}-\mathbf{P})\hat{f}), \frac{ik_j \hat{c}}{1+\vert k\vert^2}\Big)
=\partial_t \Big[  (1+t)^\sigma \Big(\Lambda_j ((\mathbf{I}-\mathbf{P})\hat{f}), \frac{ik_j \hat{c}}{1+\vert k\vert^2}\Big)\Big]\\
-\sigma (1+t)^{\sigma-1} \Big(\Lambda_j ((\mathbf{I}-\mathbf{P})\hat{f}), \frac{ik_j \hat{c}}{1+\vert k\vert^2}\Big).
\end{align*}
By this observation, all the terms we have to estimate are
\begin{align}
&\Big[ \int^T_0 (1+t)^{\sigma-1} \Big(\Lambda_j ((\mathbf{I}-\mathbf{P})\hat{f}), \frac{ik_j \hat{c}}{1+\vert k\vert^2}\Big) dt \Big]^{1/2},\label{macro time remainder1}\\
&\Big[ \int^T_0 (1+t)^{\sigma-1} \Big( \Theta_{jm} ((\mathbf{I}-\mathbf{P})\hat{f}) +2\hat{c} \delta_{jm}, \frac{k_j \hat{b}_m +k_m \hat{b}_j}{1+\vert k\vert^2}\Big) dt \Big]^{1/2},\label{macro time remainder2}\\
&\Big[ \int^T_0 (1+t)^{\sigma-1} \big(\hat{b}, \frac{ik\hat{a}}{1+\vert k\vert^2}\big) dt \Big]^{1/2}, \label{macro time remainder3}
\end{align}
and the other terms can be handled as in the same way of the proof of Lemma \ref{lem: macro a priori}. First of all, 
\eqref{macro time remainder1} is easy to estimate. Indeed, since it is the product of $(\mathbf{I}-\mathbf{P})f$ and $\mathbf{P}f$, it follows
\begin{align*}
&\int^T_0 (1+t)^{\sigma-1} \Big(\Lambda_j ((\mathbf{I}-\mathbf{P})\hat{f}), \frac{ik_j \hat{c}}{1+\vert k\vert^2}\Big) dt \\
&\le \int^T_0 (1+t)^\sigma \Big( C_\eta \frac{\Vert (\mathbf{I}-\mathbf{P})\hat{f} \Vert_{L^2_{v,D}}^2}{1+\vert k\vert^2} + \eta\frac{\vert k\vert^2}{1+\vert k\vert^2} \vert \hat{c}\vert^2 \Big)dt.
\end{align*}
Likewise, for \eqref{macro time remainder2} it holds
\begin{align*}
 \int^T_0 (1+t)^{\sigma-1} \Big( \Theta_{jm} ((\mathbf{I}-\mathbf{P})\hat{f}) +2\hat{c} \delta_{jm}, \frac{i}{1+\vert k\vert^2}(k_j \hat{b}_m +k_m \hat{b}_j)\Big) dt\\
 \le C_\eta \int^T_0 (1+t)^\sigma \frac{\Vert (\mathbf{I}-\mathbf{P})\hat{f} \Vert_{L^2_{v,D}}^2}{1+\vert k\vert^2} dt + \eta   \int^T_0 (1+t)^\sigma \frac{\vert k\vert^2}{1+\vert k\vert^2} \vert \hat{b}\vert^2 dt
\end{align*}
for small $\eta>0$.
These terms with the coefficient $\eta$ can be absorbed in the end of the proof as in that of \ref{lem: macro a priori}.

Let us consider the last term \eqref{macro time remainder3}. If $\vert k\vert \ge 1$, we easily have
\begin{align*}
&\Big[ \int^T_0 (1+t)^{\sigma-1} \Big(\hat{b}, \frac{ik\hat{a}}{1+\vert k\vert^2}\Big) dt \Big]^{1/2} 
\le \Big[ \int^T_0 (1+t)^\sigma   \frac{\vert k\vert^2 }{1+\vert k\vert^2}\vert \hat{a}\vert \vert \hat{b}\vert dt \Big]^{1/2} \\
&\le \frac{\vert k \vert}{\sqrt{1+\vert k\vert^2}} \Big[ \eta \Big(\int^T_0 (1+t)^\sigma \vert \hat{a}\vert^2 dt\Big)^{1/2}
+C_\eta \Big( \int^T_0(1+t)^\sigma \vert \hat{b}\vert^2 dt\Big)^{1/2} \Big],
\end{align*}
where the term with coefficient $\eta$ can be absorbed in the end and the other term with coefficient $C_\eta$ can be absorbed by taking the suitable combination with other energy inequalities as we have done in the proof of Lemma \ref{lem: macro a priori}.  
If $\vert k\vert\le 1$, using the same technique as in \eqref{eq:InterpolationHard}, we make the interpolation
\begin{align*}
\frac{\vert k\vert}{1+\vert k\vert^2}(1+t)^{\sigma-1} 
&= \frac{1}{1+\vert k\vert^2} (1+t)^{\sigma (1-\theta)} \vert k\vert^{2(1-\theta)} 
\cdot (1+t)^{\theta (\sigma-\omega)}  \vert k \vert^{1-2(1-\theta)}\\
&\le \frac{\eta^2}{1+\vert k\vert^2}  (1+t)^\sigma  \vert k\vert^2+\frac{C_\eta^2}{1+\vert k\vert^2} (1+t)^{-1-\varepsilon} \vert k\vert^{2-\omega},
\end{align*}
where 
$\omega$ is given in \eqref{eq:InterpolationHard}.
As before, both terms are correspondingly estimated as 
\begin{align*}
&\Big[ \int^T_0  \frac{\eta^2}{1+\vert k\vert^2}  (1+t)^\sigma  \vert k\vert^2 \vert \hat{a}\vert \vert \hat{b}\vert dt \Big]^{1/2} 
\le \frac{\eta \vert k \vert}{\sqrt{1+\vert k\vert^2}}\Big(\int^T_0 (1+t)^\sigma (\vert \hat{a}\vert^2+\vert \hat{b}\vert^2) dt\Big)^{1/2},
\end{align*}
and 
\begin{align*}
&\int_{\vert k\vert\le 1}\Big[ \int^T_0 \frac{C_\eta^2}{1+\vert k\vert^2} (1+t)^{-1-\varepsilon} \vert k\vert^{{2-\omega}} \vert \hat{a}\vert \vert \hat{b}\vert dt\Big]^{1/2} dk
\le C_\eta\Vert (a,b)\Vert_{L^p_k L^\infty_T},
\end{align*}
thanks to $p'(2-\omega)<3$ for $p>3/2$.
In sum, taking $\eta>0$ small enough, we are able to derive
\begin{align*}
&\Big \Vert (1+t)^{\frac{\sigma}{2}} \frac{\vert \nabla_x \vert}{\langle \nabla_x \rangle} (a,b,c) \Big\Vert_{L^1_k L^2_T}\\
&\le C\Vert f_0 \Vert_{L^1_k L^2_v}
+C\Vert (1+t)^{\sigma/2} f\Vert_{L^1_k L^\infty_T L^2_v}
+C\Vert (1+t)^{\sigma/2} (\mathbf{I}-\mathbf{P})f \Vert_{L^1_k L^2_T L^2_{v,D}}\\
&+C\int_{\mathbb{R}^3} \Big( \int^T_0 (1+t)^\sigma \frac{\vert(\Lambda, \Theta)(\hat{\mathbbm{r}}+\hat{\mathbbm{h}})\vert}{1+\vert k\vert^2}dt\Big)^{1/2} dk 
+C\Vert (a,b,c)\Vert_{L^p_k L^\infty_T}.
\end{align*}
Furthermore, making similar calculations of Lemmas \ref{lem: inner product Gamma} and \ref{lem: inner product L} gives
\begin{align*}
&\int_{\mathbb{R}^3} \Big( \int^T_0 (1+t)^\sigma \frac{\vert(\Lambda, \Theta)(\hat{\mathbbm{r}}+\hat{\mathbbm{h}})\vert}{1+\vert k\vert^2}dt\Big)^{1/2} dk\\
&\le C\Vert (1+t)^{\sigma/2} (\mathbf{I}-\mathbf{P})f\Vert_{L^1_k L^2_t L^2_{v,D}}
+C\int_{\mathbb{R}^3} \Big( \int^T_0 (1+t)^\sigma \frac{\vert (\hat{H}, \mu^{1/4})\vert^2}{1+\vert k\vert^2} dt \Big)^{1/2} dk.
\end{align*}
Substituting this inequality, we obtain the desired result \eqref{ineq: macro time-weighted}. The proof of Lemma \ref{lem.twemp} is complete.
\end{proof}

\begin{proof}[Proof of Theorem \ref{thm: MainHard}]
Note that \eqref{thm.hc.c2} directly follows from \eqref{pinfk}. To further show \eqref{thm.hc.c1}, combining \eqref{ineq: micro time-weighted} and \eqref{ineq: macro time-weighted}, we have
\begin{align*}
&\Vert  (1+t)^{\sigma/2} f\Vert_{L^1_k L^\infty_T L^2_v}  
+\Vert (1+t)^{\sigma/2} (\mathbf{I}-\mathbf{P})f \Vert_{L^1_k L^2_T L^2_{v,D}}+\Big \Vert (1+t)^{\frac{\sigma}{2}} \frac{\vert \nabla_x \vert}{\langle \nabla_x \rangle} (a,b,c) \Big\Vert_{L^1_k L^2_T}\\
&\le C\Vert f_0\Vert_{L^1_k L^2_v} +C\Vert f_0\Vert_{L^p_k L^2_v} 
+ C \Vert (1+t)^{\sigma/2} f\Vert_{L^1_k L^\infty_T L^2_v} \Vert f\Vert_{L^1_k L^2_T L^2_{v,D}}\\
& +C\int_{\mathbb{R}^3} \Big( \int^T_0 (1+t)^\sigma \frac{\vert (\hat{H}, \mu^{1/4})\vert^2}{1+\vert k\vert^2} dt \Big)^{1/2} dk,
\end{align*}
where we have applied \eqref{pinfk} to bound
\begin{equation}\notag
\Vert (a,b,c)\Vert_{L^p_k L^\infty_T}\leq  C\Vert \mathbf{P}f\Vert_{L^p_k L^\infty_TL^2_v}\leq C\Vert f\Vert_{L^p_k L^\infty_TL^2_v}\leq C \Vert f_0\Vert_{L^p_k L^2_v}.
\end{equation}
Following the proof of Lemma \ref{lem: inner product Gamma}, with $H=\Gamma(f,f)$ we have
\begin{align*}
&\int_{\mathbb{R}^3} \Big( \int^T_0 (1+t)^\sigma \frac{\vert (\hat{H}, \mu^{1/4})\vert^2}{1+\vert k\vert^2} dt \Big)^{1/2} dk \\
&\le C  \int_{\mathbb{R}^3} \int_{\mathbb{R}^3} \Big( \int^T_0 (1+t)^\sigma \Vert \hat{f}(k-\ell)\Vert_{L^2_v}^2 \Vert \hat{f}(\ell) \Vert_{L^2_{v,D}}^2  dt \Big)^{1/2} d\ell dk\\
&\le C \Vert (1+t)^{\sigma/2} f\Vert_{L^1_k L^\infty_T L^2_v} \Vert f\Vert_{L^1_k L^2_T L^2_{v,D}},
\end{align*}
which is further bounded by
\begin{equation*}
C\delta  \Vert (1+t)^{\sigma/2} f\Vert_{L^1_k L^\infty_T L^2_v} 
+C \Vert f_0 \Vert_{L^1_k L^2_v}  \Vert (1+t)^{\sigma/2} f\Vert_{L^1_k L^\infty_T L^2_v},
\end{equation*}
by \eqref{apaweak} and \eqref{p1kc}.
These terms are absorbed under the smallness assumption.

We then conclude from the whole estimates that
\begin{align*}
&\Vert  (1+t)^{\frac{\sigma}{2}} f \Vert_{L^1_k L^\infty_T L^2_v} 
+\Vert   (1+t)^{\frac{\sigma}{2}} (\mathbf{I}-\mathbf{P}) f \Vert_{L^1_k L^2_T L^2_{v,D}}
+\Big \Vert (1+t)^{\frac{\sigma}{2}} \frac{\vert \nabla_x \vert}{\langle \nabla_x \rangle} (a,b,c) \Big\Vert_{L^1_k L^2_T}\\
&\leq C \|f_0\|_{L^1_kL^2_v}+C  \Vert f_0 \Vert_{L^p_k L^2_v},
\end{align*}
which gives \eqref{thm.hc.c1}. 
Therefore, to close the a priori assumption \eqref{apaweak}, it suffices to assume 
\begin{align*}
 \|f_0\|_{L^1_kL^2_v}+\Vert f_0 \Vert_{L^p_k L^2_v}
\end{align*}
is sufficiently small. We then have proved the global existence and large-time behavior of solutions for the case of hard potentials. 
\end{proof}

\section{A priori estimates for the soft potential case}

In this section we shall prove Theorem \ref{thm: MainSoft} regarding the global existence and time decay of solutions for the soft potential case $\gamma+2s<0$ with a technical restriction $\gamma>\max\{-3, -3/2-2s\}$. Throughout the section, we always assume these two conditions on the collision kernel. Recall that in the hard potential case, the fact that $ \Vert \cdot \Vert_{L^2_v}\lesssim \Vert \cdot \Vert_{L^2_{\gamma/2+s}}\lesssim \Vert \cdot \Vert_{L^2_{v,D}}$ has played a crucial role. However, for the soft potential case, such property is not true and we need velocity-weighted estimates to close the a priori estimates. In the meantime, the application of the idea of \cite{KNN} based on the time-frequency interpolation becomes more complicated since the degenerate velocity variable is involved.  

We recall the weight function
\begin{align}\label{Weight Def Re}
w_{\ell,q}(v)=
\begin{cases}
\langle v \rangle^{\ell \vert \gamma+2s\vert},
&\ell>0,\ q=0\quad (\text{polynomial\ weight}),\\
e^{q\langle v\rangle/4}, 
&\ell=0,\ q>0\quad (\text{exponential\ weight}),\\
\langle v \rangle^{\ell \vert \gamma+2s\vert} e^{q\langle v\rangle/4},
&\ell>0,\ q>0\quad (\text{mixed\ weight}).
\end{cases}
\end{align}
We mainly consider the mixed weights in the following because the other weights can be recovered by virtually setting $q=0$ or $\ell=0$.

We list two technical lemmas.
\begin{lemma}[Weighted estimates for $L$]\label{wgesL}
Let $L$ be given by \eqref{LinearOp}, then we have the estimate
\begin{equation*}
\left( Lg,w^{2}_{\ell,q}g\right)_{L^2_v}\geq
\delta_q\Vert w_{\ell,q}g \Vert^2_{L^2_{v,D}}
-C\Vert g\Vert_{L^2_v({B_R})}^2,
\end{equation*}
where $\de_q$, $C>0$, and  $B_R$ denotes the closed ball in $\R^3_v$ with center at the origin and radius $R>0$.
Here $w_{\ell,q}$ is given as in \eqref{Weight Def Re}.
\end{lemma}

For the proof of this lemma, readers may refer to \cite{GS, AMUXY-2012-JFA, DLYZ-VMB} where they consider either the purely polynomial or exponential cases. Though the mixed-weight case is not considered in those papers, it is direct to modify their methods for the purpose. 
Notice that there exists $C_R$ such that $\Vert g \Vert_{L^2_v(B_R)} \le C_R \Vert g\Vert_{L^2_{v,D}}$, which we will use to bound the remainder.
\begin{lemma}[Weighted estimates for $\Gamma$]
Assume $\gamma+2s<0$. Then, for the weight function $w_{\ell,q}$ in \eqref{Weight Def Re}, it holds
\begin{multline}
\left|\left( \Gamma(f,g), w^2_{\ell,q} h\right)_{L^2_{v}}\right|
\lesssim
\\
\left\{\left\|\langle v\rangle^{\gamma/2+s}w_{\ell,q}f\right\|_{L^2_v}
\Vert g\Vert_{L^2_{v,D}}+\left\|\langle v\rangle^{\gamma/2+s} g\right\|_{L^2_v}
\Vert w_{\ell,q}f\Vert_{L^2_{v,D}}\right\}
\Vert w_{\ell,q}h\Vert_{L^2_{v,D}}
\\
+\min\left\{\left\|w_{\ell,q}f\right\|_{L^2_v}
\left\|\langle v\rangle^{\gamma/2+s}g\right\|_{L^2_v},\left\|g\right\|_{L^2_v}
\left\|\langle v\rangle^{\gamma/2+s}w_{\ell,q}f\right\|_{L^2_v}\right\}
\Vert w_{\ell,q}h\Vert_{L^2_{v,D}}
\\
+\left\|w_{\ell,q}g\right\|_{L^2_v}
\left\|\langle v\rangle^{\gamma/2+s}w_{\ell,q}f\right\|_{L^2_v}
\Vert w_{\ell,q}h\Vert_{L^2_{v,D}}.
\label{bnpbl}
\end{multline}
\end{lemma}
Proofs of this lemma can be found \cite[Lemma 2.3; Lemma 2.4]{DLYZ-VMB, FLLZ-2018}.

\subsection{$L^1_k$ and $L^p_k$ estimates with velocity weight}

Using the above lemmas, we first show an $L^1_k$-energy estimate of the macro and micro parts combined.
\begin{lemma}\label{lem: L^1 soft}
Assume \eqref{apaweak}. It holds
\begin{align*}
\mathcal{E}_1(T)+\mathcal{D}_1(T)\le C\Vert w_{\ell,q} f_0 \Vert_{L^1_k L^2_v},
\end{align*}
where the energy term $\mathcal{E}_1(T)$ and the dissipation term $\mathcal{D}_1(T)$ are defined as
\begin{align}
\mathcal{E}_1(T)&= \Vert f \Vert_{L^1_k L^\infty_T L^2_v} + \Vert w_{\ell,q} (\mathbf{I}-\mathbf{P})f \Vert_{L^1_{\vert k\vert\le 1} L^\infty_T L^2_v} + \Vert w_{\ell,q} f \Vert_{L^1_{\vert k\vert \ge 1} L^\infty_T L^2_v}, \label{eq:SoftL^1Energy}\\
\mathcal{D}_1(T) &= \Vert (\mathbf{I}-\mathbf{P})f\Vert_{L^1_k L^2_T L^2_{v,D}} 
+ \Vert w_{\ell,q} (\mathbf{I}-\mathbf{P})f \Vert_{L^1_{\vert k\vert\le 1} L^2_T L^2_{v,D}}
+  \Vert w_{\ell,q} f \Vert_{L^1_{\vert k\vert \ge 1} L^2_T L^2_{v,D}}\notag \\
&+ \Big\Vert \frac{\vert \nabla_x \vert}{\langle \nabla_x \rangle} (a,b,c) \Big\Vert_{L^1_k L^2_T},
\label{eq:SoftL^1Diss}
\end{align}
respectively.
\end{lemma}

\begin{proof}

For the macroscopic estimate, by Lemma \ref{lem: macro a priori} with $H=\Gamma(f,f)$, we already know that
\begin{align}
\Big\Vert \frac{\vert \nabla_x \vert}{\langle \nabla_x \rangle} (a,b,c) \Big\Vert_{L^1_k L^2_T} 
&\le C \Vert f_0\Vert_{L^1_k L^2_v} + C(\Vert f\Vert_{L^1_k L^\infty_T L^2_v} + \Vert (\mathbf{I}-\mathbf{P}) f \Vert_{L^1_k L^2_T L^2_{v,D}} )\notag \\
&+ C\Vert f \Vert_{L^1_k L^\infty_T L^2_v} \Vert f\Vert_{L^1_k L^2_T L^2_{v,D}}. \label{ineq: P low freq}
\end{align}
Also, recall that we have proved the non-weighted $L^1_k$ estimate \eqref{p1kcc}.
Following~\cite{Str12}, we shall deduce velocity-weighted estimates of the micro part under
\eqref{apaweak}, and combine them with \eqref{ineq: P low freq} and \eqref{p1kcc} to show the desired result.

Indeed, we first apply $(\mathbf{I}-\mathbf{P})$ to \eqref{eq: fourier} to have
\begin{align}
&\partial_t (\mathbf{I}-\mathbf{P}) \hat{f}+iv\cdot k (\mathbf{I}-\mathbf{P}) \hat{f} +L(\mathbf{I}-\mathbf{P}) \hat{f} \notag\\
&=\widehat{\Gamma(f,f)}-(\mathbf{I}-\mathbf{P})[iv\cdot k \mathbf{P}\hat{f}] +\mathbf{P}[iv\cdot k (\mathbf{I}-\mathbf{P}) \hat{f}]. 
\label{lem.l1k.pad1}
\end{align}
We multiply this with $w_{\ell,q}^2 (\mathbf{I}-\mathbf{P}) \bar{\hat{f}}$ to have
\begin{align*}
\frac{1}{2} \frac{\partial}{\partial t}\vert w_{\ell,q} (\mathbf{I}-\mathbf{P}) \hat{f}\vert^2 
+\mathrm{Re}\, ( w_{\ell,q}^2 L (\mathbf{I}-\mathbf{P})\hat{f},  (\mathbf{I}-\mathbf{P})\hat{f}) = \Gamma_1 +\Gamma_2, 
\end{align*}
where $\Gamma_1=\mathrm{Re}(\widehat{\Gamma(f,f)}, w_{\ell,q}^2(\mathbf{I}-\mathbf{P})\hat{f})$ and $\Gamma_2$ is a similar inner product term involving all the other terms of the right hand side of \eqref{lem.l1k.pad1}. We integrate the above identity over $\mathbb{R}^3_v \times [0,T]$ and then take the square root to have
\begin{multline}
\Vert w_{\ell,q} (\mathbf{I}-\mathbf{P}) \hat{f} \Vert_{L^\infty_T L^2_v} + \Vert w_{\ell,q} (\mathbf{I}-\mathbf{P}) \hat{f} \Vert_{L^2_T L^2_{v,D}} 
\le  C\Vert w_{\ell,q} (\mathbf{I}-\mathbf{P}) \hat{f}_0 \Vert_{L^2_v}\\
+ C\Vert (\mathbf{I}-\mathbf{P}) \hat{f} \Vert_{L^2_T L^2_v(B_R)} 
+ C\Big(\int^T_0 \Big\vert \int_{\mathbb{R}^3}  (\Gamma_1 + \Gamma_2)  dv\Big\vert dt\Big)^{1/2}, \label{ineq: weight I-P 1} 
\end{multline}
where Lemma \ref{wgesL} has been used.
Since
\begin{align*}
\Big\vert \int_{\mathbb{R}^3} \Gamma_2 dv\Big\vert \le \eta^2 \Vert w_{\ell,q} (\mathbf{I}-\mathbf{P}) \hat{f} \Vert_{L^2_{v,D}}^2 + C_\eta^2 \vert k\vert^2 \big( \Vert \langle v\rangle^{-m} (\mathbf{I}-\mathbf{P}) \hat{f} \Vert_{L^2_v}^2+\Vert \mathbf{P}\hat{f}\Vert_{L^2_v}^2\big)
\end{align*} 
for any $m\ge 0$, where $\eta>0$ can be small enough, we have 
\begin{align*}
\Big(\int^T_0 \Big\vert \int_{\mathbb{R}^3} \Gamma_2 dv\Big\vert dt\Big)^{1/2}
\le \eta \Vert w_{\ell,q} (\mathbf{I}-\mathbf{P}) \hat{f} \Vert_{L^2_T L^2_{v,D}} + C_\eta \vert k\vert \Vert \hat{f}\Vert_{L^2_T L^2_{v,D}}
\end{align*}
because $\Vert \langle v\rangle^{-m} (\mathbf{I}-\mathbf{P}) \hat{f} \Vert_{L^2_v}\le C \Vert (\mathbf{I}-\mathbf{P}) \hat{f} \Vert_{L^2_{v,D}}$ for sufficiently large $m$ and $\Vert  \mathbf{P}\hat{f}\Vert_{L^2_v}\le C \Vert \hat{f}\Vert_{L^2_{v,D}}$. 
Also, \eqref{bnpbl} implies
\begin{align*}
\Big( \int^T_0 \Big\vert \int_{\mathbb{R}^3_v} \Gamma_1  dv \Big\vert dt\Big)^{1/2}
&\le \eta \Vert w_{\ell,q} (\mathbf{I}-\mathbf{P})\hat{f} \Vert_{L^2_T L^2_{v,D}}\\
& +C_\eta \Vert w_{\ell,q} \hat{f} \Vert_{L^\infty_T L^2_v} *_k \Vert w_{\ell,q} \hat{f}\Vert_{L^2_T L^2_{v,D}}(k).
\end{align*}
Substituting these estimates into \eqref{ineq: weight I-P 1} and then integrating the resultant over $\vert k\vert\le 1$, it holds
\begin{align}
&\Vert w_{\ell,q} (\mathbf{I}-\mathbf{P}) f \Vert_{L^1_{\vert k\vert\le 1} L^\infty_T L^2_v} 
+ \Vert w_{\ell,q} (\mathbf{I}-\mathbf{P}) f \Vert_{L^1_{\vert k\vert\le 1} L^2_T L^2_{v,D}} \notag \\
&\le  C\Vert w_{\ell,q} (\mathbf{I}-\mathbf{P}) f_0 \Vert_{L^1_{\vert k\vert\le 1} L^2_v}
+ C\Vert (\mathbf{I}-\mathbf{P}) f \Vert_{L^1_{\vert k\vert\le 1} L^2_T L^2_v(B_R)} 
+ C \Vert \vert \nabla_x \vert f \Vert_{L^1_{\vert k\vert\le 1} L^2_T L^2_{v,D}} \notag \\
&+C \Vert w_{\ell,q} f \Vert_{L^1_k L^\infty_T L^2_v} \Vert w_{\ell,q} f\Vert_{ L^1_k L^2_T L^2_{v,D}}. \label{ineq: weight I-P low freq}
\end{align}
Notice we have used the inequality
\begin{align*}
\int_{\vert k\vert\le 1} \int_{\mathbb{R}^3_\ell} \vert \hat{F}(k-\ell)\vert \vert \hat{G}(\ell) \vert d\ell dk
\le \Vert F\Vert_{L^1_k} \Vert G\Vert_{L^1_k}
\end{align*}
for the nonlinear term.
Likewise, by a similar procedure without applying $(\mathbf{I}-\mathbf{P})$ first and then integrating over $\vert k\vert \ge 1$, one can show
\begin{align}
&\Vert w_{\ell,q} f \Vert_{L^1_{\vert k\vert \ge 1} L^\infty_T L^2_v} + \Vert w_{\ell,q} f \Vert_{L^1_{\vert k\vert \ge 1} L^2_T L^2_{v,D}} \notag \\
&\le C \Vert w_{\ell,q} f _0 \Vert_{L^1_k L^2_v} +C\Vert f \Vert_{L^1_{\vert k\vert \ge 1} L^2_T L^2_v(B_R)}\notag \\
&+ C \Vert w_{\ell,q} f \Vert_{L^1_k L^\infty_T L^2_v} \Vert w_{\ell,q} f\Vert_{ L^1_k L^2_T L^2_{v,D}}. \label{ineq: weight high freq}
\end{align}

Now, consider the linear combination \eqref{pinfk}$+\varepsilon_0$\eqref{ineq: P low freq}$+\varepsilon_1$\eqref{ineq: weight I-P low freq}$+\varepsilon_2$\eqref{ineq: weight high freq} for suitable constants $0< \varepsilon_2\ll \varepsilon_1\ll \varepsilon_0\ll 1$,
where the left-hand side of the resultant is equivalent to $\mathcal{E}_1(T)+ \mathcal{D}_1(T)$.
The linear terms on the right-hand side of the resultant are absorbed into the left-hand side as follows: First, in \eqref{ineq: weight I-P low freq}, we have
\begin{align*}
\Vert \vert \nabla_x \vert f\Vert_{L^1_{\vert k\vert\le 1} L^2_T L^2_{v,D}} 
\le C \Vert (\mathbf{I}-\mathbf{P})f\Vert_{L^1_k L^2_T L^2_{v,D}}
+
C\Big\Vert \frac{\vert \nabla_x \vert}{\langle \nabla_x \rangle} (a,b,c) \Big\Vert_{L^1_k L^2_T}
\le C \mathcal{D}_1(T)
\end{align*}
because 
\begin{align*}
\vert k\vert \sim \frac{\vert k\vert}{\sqrt{1+\vert k\vert^2}}\le 1 \quad
\text{for}\quad \vert k\vert \le 1.
\end{align*}
The equivalence is used for the macro part, and the inequality is used for the micro part.
Also in \eqref{ineq: weight I-P low freq}, it holds
\begin{align*}
\Vert (\mathbf{I}-\mathbf{P})f \Vert_{L^1_{\vert k\vert\le 1} L^2_T L^2_v(B_R)}
\le C \Vert (\mathbf{I}-\mathbf{P})f \Vert_{L^1_{\vert k\vert\le 1} L^2_T L^2_{v,D}}
\le C \mathcal{D}_1(T),
\end{align*}
owing to $\Vert \cdot \Vert_{L^2_v(B_R)}\le C_R \Vert \cdot \Vert_{L^2_{v,D}}$. By the same reason and $1\le \frac{2\vert k\vert^2}{1+\vert k\vert^2}$ for $\vert k\vert \ge 1$, in \eqref{ineq: weight high freq} it holds
\begin{align*}
\Vert f\Vert_{L^1_{\vert k \vert \ge 1} L^2_T L^2_v(B_R)}
\le C \Vert (\mathbf{I}-\mathbf{P})f\Vert_{L^1_k L^2_T L^2_{v,D}}
+ C \Big\Vert \frac{\vert \nabla_x \vert}{\langle \nabla_x \rangle} (a,b,c) \Big\Vert_{L^1_k L^2_T}
\le C\mathcal{D}_1(T).
\end{align*}
Other linear terms coming from \eqref{p1kcc} and \eqref{ineq: P low freq} can be handled as in the hard potential case, provided that $\varepsilon_0$, $\varepsilon_1$, and $\varepsilon_2$ are small.

Next, we focus on estimates of the nonlinear term 
\begin{align*}
N(f,f):=\Vert w_{\ell,q} f \Vert_{L^1_k L^\infty_T L^2_v} \Vert w_{\ell,q} f\Vert_{ L^1_k L^2_T L^2_{v,D}}
\end{align*}
under the assumption
\eqref{apaweak}. By the macro-micro decomposition, we divide it into four terms
\begin{align*}
N(f,f)&\le N(\mathbf{P}f, \mathbf{P}f)
+N((\mathbf{I}-\mathbf{P})f, \mathbf{P}f)\\
&+N(\mathbf{P}f, (\mathbf{I}-\mathbf{P})f)
+N((\mathbf{I}-\mathbf{P})f, (\mathbf{I}-\mathbf{P})f)\\
&=:N_1+N_2+N_3+N_4.
\end{align*}

For $N_1$ and $N_2$, we can proceed as in the proof of Lemma~\ref{lem: micro est hard} because
\begin{align*}
\Vert w_{\ell,q} \mathbf{P} f\Vert_{L^1_k L^2_T L^2_{v,D}}
\le C \Vert (a,b,c)\Vert_{L^1_k L^2_T}
\le C\delta
\end{align*}
under \eqref{apaweak}. With the fact $\Vert w_{\ell,q}  \mathbf{P} F \Vert_{L^2_v}\le C \Vert F\Vert_{L^2_v}$, this yields
\begin{align*}
N_1 =\Vert w_{\ell,q}  \mathbf{P} f \Vert_{L^1_k L^\infty_T L^2_v} \Vert w_{\ell,q}  \mathbf{P} f \Vert_{L^1_k L^2_T L^2_{v,D}}
\le C\delta  \Vert f \Vert_{L^1_k L^\infty_T L^2_v}
\le C \delta \mathcal{E}_1(T).
\end{align*}
Also, since
\begin{align*}
&\Vert w_{\ell,q}   (\mathbf{I}-\mathbf{P}) f \Vert_{L^1_{\vert k\vert \ge 1} L^\infty_T L^2_v} \notag \\
&\le \Vert (\mathbf{I}-\mathbf{P}) w_{\ell,q} f \Vert_{L^1_{\vert k\vert \ge 1} L^\infty_T L^2_v}
+\Vert [w_{\ell,q}, (\mathbf{I}-\mathbf{P})]  f \Vert_{L^1_{\vert k\vert \ge 1} L^\infty_T L^2_v} \notag \\
&=\Vert (\mathbf{I}-\mathbf{P}) w_{\ell,q} f \Vert_{L^1_{\vert k\vert \ge 1} L^\infty_T L^2_v}
+\Vert [w_{\ell,q}, \mathbf{P}]  f \Vert_{L^1_{\vert k\vert \ge 1} L^\infty_T L^2_v} \notag \\
&\le C\Vert w_{\ell,q} f \Vert_{L^1_{\vert k\vert \ge 1} L^\infty_T L^2_v}
+C \Vert f \Vert_{L^1_{\vert k\vert \ge 1} L^\infty_T L^2_v} \notag \\
&\le C \mathcal{E}_1(T),
\end{align*}
where $[\cdot,\cdot]$ is the commutator,
we have
\begin{align}
\Vert w_{\ell,q}   (\mathbf{I}-\mathbf{P}) f \Vert_{L^1_k L^\infty_T L^2_v} 
&=\Vert w_{\ell,q}   (\mathbf{I}-\mathbf{P}) f \Vert_{L^1_{\vert k\vert\le 1} L^\infty_T L^2_v} 
+\Vert w_{\ell,q}   (\mathbf{I}-\mathbf{P}) f \Vert_{L^1_{\vert k\vert\ge 1} L^\infty_T L^2_v} \notag \\
&\le C\mathcal{E}_1(T) \label{ineq:I-P energy bound}
\end{align} 
and thus
\begin{align*}
N_2 = \Vert w_{\ell,q}  (\mathbf{I}-\mathbf{P}) f \Vert_{L^1_k L^\infty_T L^2_v} \Vert w_{\ell,q}  \mathbf{P} f \Vert_{L^1_k L^2_T L^2_{v,D}}
\le C\delta \mathcal{E}_1(T).
\end{align*}

Similar to the estimate of $N_1$, we have
\begin{align*}
N_3 &= \Vert w_{\ell,q} \mathbf{P}   f \Vert_{L^1_k L^\infty_T L^2_v} \Vert w_{\ell,q}  (\mathbf{I}-\mathbf{P})f \Vert_{L^1_k L^2_T L^2_{v,D}}\\
&\le C \Vert f \Vert_{L^1_k L^\infty_T L^2_v} \Vert w_{\ell,q}  (\mathbf{I}-\mathbf{P})f \Vert_{L^1_k L^2_T L^2_{v,D}}
\le C \mathcal{E}_1(T) \mathcal{D}_1(T).
\end{align*}
Notice we can show $\Vert w_{\ell,q}  (\mathbf{I}-\mathbf{P})f \Vert_{L^1_k L^2_T L^2_{v,D}}\le C\mathcal{D}_1(T)$ as \eqref{ineq:I-P energy bound} is shown.

Finally, the two estimates for $w_{\ell,q} (\mathbf{I}-\mathbf{P}) f$ readily give
\begin{align*}
N_4=\Vert w_{\ell,q}  (\mathbf{I}-\mathbf{P})f \Vert_{L^1_k L^\infty_T L^2_v} \Vert w_{\ell,q}  (\mathbf{I}-\mathbf{P})f \Vert_{L^1_k L^2_T L^2_{v,D}}
\le C\mathcal{E}_1(T) \mathcal{D}_1(T).
\end{align*}

Likewise, for the non-weighted part, it holds
\begin{align*}
\Vert f \Vert_{L^1_kL^\infty_T L^2_v} \Vert f\Vert_{ L^1_k L^2_T L^2_{v,D}} 
&\le \Vert f \Vert_{L^1_k L^\infty_T L^2_v} \Vert (\mathbf{I}-\mathbf{P}) f\Vert_{ L^1_k L^2_T L^2_{v,D}}
 +C\delta \Vert f \Vert_{L^1_k L^\infty_T L^2_v}\\
&\le C \delta \mathcal{E}_1(T)+C \mathcal{E}_1(T) \mathcal{D}_1(T).
\end{align*}

Hence, if $\delta$ is small, the substitution of these inequalities into the aforementioned linear combination
gives
\begin{align*}
\mathcal{E}_1(T)+\mathcal{D}_1(T)\le C\mathcal{E}_1(0)+C\mathcal{E}_1(T)\mathcal{D}_1(T).
\end{align*}
Since $\mathcal{E}_1(0)\simeq\Vert w_{\ell,q}f_0\Vert_{L^1_k L^2_v}$ is sufficiently small,
this gives the desired estimate. 
\end{proof}

By a similar argument, we can also show a similar estimate in $L^p_k$. 

\begin{lemma}\label{lem: L^p soft}
Let $3/2<p\le \infty$.
Under the same assumption as in Lemma \ref{lem: L^1 soft}, we have 
\begin{align*}
\mathcal{E}_p(T) + \mathcal{D}_p(T) \le C \Vert w_{\ell,q} f_0 \Vert_{L^p_k L^2_v},
\end{align*}
where
\begin{align}
\mathcal{E}_p(T)&=\Vert f \Vert_{L^p_k L^\infty_T L^2_v} 
+ \Vert w_{\ell,q} (\mathbf{I}-\mathbf{P})f \Vert_{L^p_{\vert k\vert\le 1} L^\infty_T L^2_v} 
+ \Vert w_{\ell,q} f \Vert_{L^p_{\vert k\vert \ge 1} L^\infty_T L^2_v}, \label{eq:SoftL^pEnergy}\\
\mathcal{D}_p(T)&= \Vert (\mathbf{I}-\mathbf{P})f\Vert_{L^p_k L^2_T L^2_{v,D}} 
+ \Vert w_{\ell,q} (\mathbf{I}-\mathbf{P})f \Vert_{L^p_{\vert k\vert\le 1} L^2_T L^2_{v,D}} \notag\\
&+ \Vert w_{\ell,q} f \Vert_{L^p_{\vert k\vert \ge 1} L^2_T L^2_{v,D}}
+ \Big\Vert \frac{\vert \nabla_x \vert}{\langle \nabla_x \rangle} (a,b,c) \Big\Vert_{L^p_k L^2_T}.
\label{eq:SoftL^pDiss}
\end{align}
\end{lemma}

\begin{proof}
Recall $\Vert fg \Vert_{L^p_k}\le \Vert f\Vert_{L^p_k} \Vert g\Vert_{L^1_k}$. Repeating the same argument as in the proof of Lemma \ref{lem: L^1 soft} while using $L^p_k$ instead of $L^1_k$, then applying this inequality to nonlinear terms, we can show
\begin{align}
\mathcal{E}_p(T)+ \mathcal{D}_p(T)
&\le C \Vert w_{\ell,q}  f_0\Vert_{L^p_k L^2_v} +C \Vert w_{\ell,q} (\mathbf{I}-\mathbf{P}) f_0 \Vert_{L^p_{\vert k\vert \le 1} L^2_v} \notag \\
&+ C\varepsilon_0 \Vert f \Vert_{L^p_k L^\infty_T L^2_v} \Vert f\Vert_{L^1_k L^2_T L^2_{v,D}}\notag \\
&+ C (\varepsilon_1+\varepsilon_2) \Vert w_{\ell,q} f \Vert_{L^p_k L^\infty_T L^2_v} \Vert w_{\ell,q} f\Vert_{ L^1_k L^2_T L^2_{v,D}}. \notag
\end{align}
Notice all the linear terms can be absorbed into the left-hand side exactly in the same manner as the previous proof.

Here, the nonlinear terms on the right-hand side can be dealt with as follows. We first have
\begin{align*}
\Vert w_{\ell,q} f \Vert_{L^p_k L^\infty_T L^2_v} \Vert w_{\ell,q} f\Vert_{ L^1_k L^2_T L^2_{v,D}}
&\le \Vert w_{\ell,q} f \Vert_{L^p_k L^\infty_T L^2_v} \Vert w_{\ell,q} \mathbf{P} f\Vert_{ L^1_k L^2_T L^2_{v,D}} \\
&+ \Vert w_{\ell,q} f \Vert_{L^p_k L^\infty_T L^2_v} \Vert w_{\ell,q} (\mathbf{I}-\mathbf{P}) f\Vert_{ L^1_k L^2_T L^2_{v,D}}.
\end{align*}
We already know that $\Vert  w_{\ell,q} (\mathbf{I}-\mathbf{P}) f\Vert_{ L^1_k L^2_T L^2_{v,D}} \le C\mathcal{D}_1(T)\le C \Vert w_{\ell,q} f_0 \Vert_{L^1_k L^2_v}$ by the previous lemma, and $\Vert w_{\ell,q} \mathbf{P} f\Vert_{ L^1_k L^2_T L^2_{v,D}}\le C\delta$ for small $q>0$ by
\eqref{apaweak}. 
Therefore the sum of these terms is bounded by 
\begin{align*}
C(\delta+\Vert w_{\ell,q} f_0 \Vert_{L^1_k L^2_v}) \Vert w_{\ell,q} f \Vert_{L^p_k L^\infty_T L^2_v}
\end{align*}
if $\delta$ and the initial data are sufficiently small. Further, as in the $L^1_k$-case, we can show $\Vert w_{\ell,q} f \Vert_{L^p_k L^\infty_T L^2_v}\le C \mathcal{E}_p(T)$ by the macro-micro and frequency decompositions.
In the same manner,
\begin{align*}
\Vert f \Vert_{L^p_k L^\infty_T L^2_v} \Vert f\Vert_{L^1_k L^2_T L^2_{v,D}}
\le C(\delta + \Vert w_{\ell,q} f_0\Vert_{L^1_k L^2_v}) \mathcal{E}_p(T).
\end{align*} 
Therefore, if $\delta +\Vert w_{\ell,q} f_0\Vert_{L^1_k L^2_v}$ are sufficiently small, the nonlinear terms above can be absorbed into the left-hand side, which completes the proof.
\end{proof}

\subsection{Time-weighted estimates}

Based on Lemmas \ref{lem: L^1 soft} and \ref{lem: L^p soft}, we will be able to deduce the time-weighted estimates similar to the hard potential case. To overcome the degeneration of dissipation in large velocity for soft potentials, we have to carry out the interpolation in triple variables $t$-$k$-$v$. We use the splitting idea in \cite{SG-CPDE,SG-08-ARMA} to treat the $t$-$v$ variables while we use the idea in \cite{KNN} to deal with the frequency variable $k$.

\begin{lemma}\label{lem: time-weighted soft}
Assume \eqref{ineq:SoftIndexCondition}, \eqref{MainSoftInitial}, and \eqref{apaweak}, then it holds
\begin{align*}
\mathscr{E}(T)+\mathscr{D}(T)
\le C \Vert w_{\ell+j,q}f_0\Vert_{L^1_k L^2_v} 
+C \Vert w_{\ell+j,q}f_0\Vert_{L^p_k L^2_v},
\end{align*}
where
\begin{align*}
\mathscr{E}(T)
&=\Vert (1+t)^{\sigma/2} f\Vert_{L^1_k L^\infty_T L^2_v}
+\Vert (1+t)^{\sigma/2} w_{\ell,q} (\mathbf{I}-\mathbf{P}) f\Vert_{L^1_{\vert k\vert \le 1} L^\infty_T L^2_v}\\
&+\Vert (1+t)^{\sigma/2} w_{\ell,q} f\Vert_{L^1_{\vert k \vert \ge 1} L^\infty_T L^2_v},\\
\mathscr{D}(T)
&= \Vert (1+t)^{\sigma/2} (\mathbf{I}-\mathbf{P}) f \Vert_{L^1_k L^2_T L^2_{v,D}}
+\Vert (1+t)^{\sigma/2} w_{\ell,q} (\mathbf{I}-\mathbf{P})  f \Vert_{L^1_{\vert k\vert \le 1} L^2_T L^2_{v,D}}\\
&+\Vert (1+t)^{\sigma/2} w_{\ell,q}  f \Vert_{L^1_{\vert k\vert \ge 1} L^2_T L^2_{v,D}}
+\Big\Vert \frac{\vert \nabla_x \vert}{\langle \nabla_x \rangle} (1+t)^{\sigma/2} (a,b,c) \Big\Vert_{L^1_k L^2_T}.
\end{align*}
\end{lemma}

\begin{proof}
We multiply $(1+t)^{\sigma/2}$ with \eqref{eq: fourier} and carry out the same argument as in the proof of Lemma \ref{lem: L^1 soft}.
Then instead of \eqref{ineq: weight I-P low freq}, one can show
\begin{align}
&\Vert (1+t)^{\sigma/2} w_{\ell,q} (\mathbf{I}-\mathbf{P}) f \Vert_{L^1_{\vert k\vert\le 1} L^\infty_T L^2_v} 
+ \Vert (1+t)^{\sigma/2} w_{\ell,q} (\mathbf{I}-\mathbf{P}) f \Vert_{L^1_{\vert k\vert\le 1} L^2_T L^2_{v,D}} \notag \\
&\le  C\Vert w_{\ell,q} (\mathbf{I}-\mathbf{P}) f_0 \Vert_{L^1_{\vert k\vert\le 1} L^2_v}
+ C \Vert (1+t)^{\sigma/2} (\mathbf{I}-\mathbf{P}) f \Vert_{L^1_{\vert k\vert\le 1} L^2_T L^2_v(B_R)} \notag\\
&+ C \Vert (1+t)^{\sigma/2} \vert \nabla_x \vert f \Vert_{L^1_{\vert k\vert\le 1} L^2_T L^2_{v,D}}\notag \\
&+C \Vert (1+t)^{\sigma/2} w_{\ell,q} f \Vert_{L^1_kL^\infty_T L^2_v} \Vert w_{\ell,q} f\Vert_{ L^1_k L^2_T L^2_{v,D}} \notag\\
&+  C\sqrt{\sigma} \int_{\vert k\vert\le 1} \Big(\int^T_0 (1+t)^{\sigma-1}\vert w_{\ell,q} (\mathbf{I}-\mathbf{P}) \hat{f} \vert_{L^2_v}^2 dt \Big)^{1/2} dk. 
\label{ineq: time weight I-P low freq}
\end{align}
Also, instead of \eqref{ineq: weight high freq} one can show
\begin{align}
&\Vert (1+t)^{\sigma/2} w_{\ell,q} f \Vert_{L^1_{\vert k\vert \ge 1} L^\infty_T L^2_v} + \Vert (1+t)^{\sigma/2} w_{\ell,q} f \Vert_{L^1_{\vert k\vert \ge 1} L^2_T L^2_{v,D}} \notag \\
&\le C \Vert w_{\ell,q} f _0 \Vert_{L^1_k L^2_v} +C\Vert (1+t)^{\sigma/2} f \Vert_{L^1_{\vert k\vert \ge 1} L^2_T L^2_v(B_R)}\notag\\
&+ C \Vert (1+t)^{\sigma/2} w_{\ell,q} f \Vert_{L^1_k L^\infty_T L^2_v} \Vert w_{\ell,q} f\Vert_{ L^1_k L^2_T L^2_{v,D}}\notag\\
&+  C \sqrt{\sigma} \int_{\vert k\vert \ge 1} \Big(\int^T_0 (1+t)^{\sigma-1} \Vert w_{\ell,q} \hat{f} \Vert_{L^2_v}^2 dt \Big)^{1/2} dk.
\label{ineq: time weight high freq}
\end{align}
As in the hard potential case, it can be verified that
\begin{align}
&\Big\Vert \frac{\vert \nabla_x \vert}{\langle \nabla_x \rangle} (1+t)^{\sigma/2} (a,b,c) \Big\Vert_{L^1_k L^2_T} 
\le C \Vert f_0 \Vert_{L^1_k L^2_v} 
+C \Vert (1+t)^{\sigma/2} f\Vert_{L^1_k L^\infty_T L^2_v}\notag \\
&+C\Vert (1+t)^{\sigma/2}(\mathbf{I}-\mathbf{P})f \Vert_{L^1_k L^2_T L^2_{v,D}} 
+C\Vert (a,b,c)\Vert_{L^p_k L^\infty_T}\notag \\
&+C \Vert (1+t)^{\sigma/2} f \Vert_{L^1_k L^\infty_T L^2_v} \Vert f\Vert_{L^1_k L^2_T L^2_{v,D}}
\label{ineq: time weight abc}
\end{align}
and
\begin{align}
&\Vert (1+t)^{\sigma/2} f \Vert_{L^1_k L^\infty_T L^2_v} + \Vert (1+t)^{\sigma/2} (\mathbf{I}-\mathbf{P})f \Vert_{L^1_k L^2_T L^2_{v,D}}
\le C \Vert f_0\Vert_{L^1_k L^2_v}\notag \\
&+C \Vert (1+t)^{\sigma/2} f \Vert_{L^1_k L^\infty_T L^2_v} \Vert f\Vert_{L^1_k L^2_T L^2_{v,D}}
+C\sqrt{\sigma}\int_{\mathbb{R}^3} \Big(\int^T_0 (1+t)^{\sigma-1} \Vert \hat{f} \Vert_{L^2_v}^2 dt \Big)^{1/2} dk.
\label{ineq: time weight no velocity}
\end{align}
Therefore, similar to the proofs of Lemmas \ref{lem: micro time-weighted} and \ref{lem.twemp}, once we can show appropriate estimates of the following three terms 
\begin{gather}
\int_{\vert k\vert\le 1} \Big(\int^T_0 (1+t)^{\sigma-1}\Vert w_{\ell,q} (\mathbf{I}-\mathbf{P}) \hat{f} \Vert_{L^2_v}^2 dt \Big)^{1/2} dk, \label{time-remainder-soft1}\\
\int_{\vert k\vert \ge 1} \Big(\int^T_0 (1+t)^{\sigma-1} \Vert w_{\ell,q} \hat{f} \Vert_{L^2_v}^2 dt \Big)^{1/2} dk, \label{time-remainder-soft2}\\
\int_{\mathbb{R}^3} \Big(\int^T_0 (1+t)^{\sigma-1} \Vert \hat{f}\Vert_{L^2_v}^2 dt \Big)^{1/2} dk \label{time-remainder-soft3},
\end{gather}
we can close the a priori estimate by an appropriate linear combination of \eqref{ineq: time weight I-P low freq}, \eqref{ineq: time weight high freq}, \eqref{ineq: time weight abc} and \eqref{ineq: time weight no velocity}.

As in \cite{SG-CPDE,SG-08-ARMA}, we decompose the full velocity space $\mathbb{R}^3_v$ into  $E+E^c$ with
\begin{align}\label{EandEC}
E=\{ \langle v\rangle^{\vert \gamma+2s\vert} \le (1+t)^{r}\},\quad
{0<r<\frac{p'\varepsilon}{3+p'}\in \Big[\frac{\varepsilon}{4}, \frac{\varepsilon}{2}\Big)}
\end{align}
for $3/2<p\le \infty$, where $p'$ is the H\"older conjugate of $p$.
We carry on estimates on each domain.

\medskip 
\noindent\textbf{Step 1:
Estimates on $E$}. 
It holds $1\le (1+t)^{r}\langle v\rangle^{\gamma+2s}$ on $E$, so
the $E$-part of \eqref{time-remainder-soft1} is bounded by
\begin{align*}
\int_{\vert k\vert\le 1} \Big(\int^T_0 \int_E (1+t)^{\sigma-1+r} \langle v\rangle^{\gamma+2s} \vert w_{\ell,q} (\mathbf{I}-\mathbf{P}) \hat{f} \vert^2 dv dt \Big)^{1/2} dk\\
\le \int_{\vert k\vert\le 1} \Big(\int^T_0 (1+t)^{\sigma-1+r} \Vert w_{\ell,q} (\mathbf{I}-\mathbf{P}) \hat{f} \Vert_{L^2_{v,D}}^2 dt \Big)^{1/2} dk
\end{align*}
since $\Vert \cdot \Vert_{L^2_{\gamma/2+s}}\le C\Vert \cdot \Vert_{L^2_{v,D}}$.
We use the interpolation with $\theta=\frac{1-r}{\sigma+\varepsilon+1-r}\in(0,1)$,
\begin{align}\label{ineq:InterpolationE}
(1+t)^{\sigma-1+r} 
= (1+t)^{\sigma(1-\theta)+\theta(-1-\varepsilon+r)}
\le \eta^2 (1+t)^\sigma + C_\eta^2 (1+t)^{-1-\varepsilon+r},
\end{align}
where $\eta >0$ is an arbitrarily small constant again.
By substitution we have
\begin{align*}
&\int_{\vert k\vert\le 1} \Big(\int^T_0 \int_E (1+t)^{\sigma-1+r} \langle v\rangle^{\gamma+2s} \vert w_{\ell,q} (\mathbf{I}-\mathbf{P}) \hat{f} \vert^2 dv dt \Big)^{1/2} dk\\
&\le \eta \int_{\vert k\vert\le 1} \Big(\int^T_0 (1+t)^{\sigma} \Vert  w_{\ell,q} (\mathbf{I}-\mathbf{P}) \hat{f} \Vert^2_{L^2_{v,D}} dv dt \Big)^{1/2} dk\\
&+ C_\eta \int_{\vert k\vert\le 1} \Big(\int^T_0 (1+t)^{-1-\varepsilon+r} \Vert  w_{\ell,q} (\mathbf{I}-\mathbf{P}) \hat{f} \Vert^2_{L^2_{v,D}} dv dt \Big)^{1/2} dk.
\end{align*}
The first term $\eta \Vert (1+t)^{\sigma/2} w_{\ell,q} (\mathbf{I}-\mathbf{P}) f \Vert_{L^1_{\vert k \vert\le 1} L^2_T L^2_{v,D}}$ is absorbed into the left hand side. Since 
$-1-\varepsilon+r<-1$, the second one is bounded by
\begin{align*}
&\Vert w_{\ell,q} (\mathbf{I}-\mathbf{P}) f \Vert_{L^p_{\vert k \vert\le 1} L^\infty_T L^2_{v,D}} \Big(\int_{\vert k\vert\le 1} dk\Big)^{1/p'} \Big( \int^T_0 (1+t)^{-1-\varepsilon+r}dt\Big)^{1/2} \\
&\le C \Vert w_{\ell,q}  f_0 \Vert_{L^p_k  L^2_v} \le C \Vert w_{\ell+j,q}  f_0 \Vert_{L^p_k  L^2_v}
\end{align*}
by the H\"older inequality and Lemma \ref{lem: L^p soft}.
Recall \eqref{Weight Def Re} for $w_{\ell,q} \le w_{\ell+j,q}$.

In the same manner, for \eqref{time-remainder-soft2}, it holds
\begin{align*}
&\int_{\vert k\vert \ge 1} \Big(\int^T_0 \int_E (1+t)^{\sigma-1} \vert w_{\ell,q} \hat{f} \vert^2 dv dt \Big)^{1/2} dk \\
&\le \int_{\vert k\vert \ge 1} \Big(\int^T_0 (1+t)^{\sigma-1+r} \Vert w_{\ell,q} \hat{f} \Vert_{L^2_{v,D}}^2 dt \Big)^{1/2} dk\\
&\le \int_{\vert k\vert \ge 1} \Big(\int^T_0 \big( \eta^2(1+t)^{\sigma}+C_\eta^2)
\Vert w_{\ell,q} \hat{f} \Vert_{L^2_{v,D}}^2 dt \Big)^{1/2} dk\\
&\le \eta\Vert (1+t)^{\sigma/2} w_{\ell,q} f \Vert_{L^1_{\vert k \vert \ge 1} L^2_T L^2_{v,D}} + C_\eta  \Vert w_{\ell,q} f_0 \Vert_{L^1_k L^2_v}.
\end{align*}

For \eqref{time-remainder-soft3}, we use the different interpolation
\begin{align*}
(1+t)^{\sigma-1+r}
&=(1+t)^{(1-\theta)\sigma} \vert k\vert^{2(1-\theta)} \cdot (1+t)^{\theta(\sigma-\omega)} \vert k\vert^{-2(1-\theta)} \\
&\le \eta^2 (1+t)^\sigma \vert k\vert^2 
+ C_\eta^2 (1+t)^{\sigma-\omega} \vert k\vert^{-2\frac{1-\theta}{\theta}}
\end{align*}
for $\vert k\vert \le 1$, where
\begin{align*}
\omega=4-\frac{3}{p}-\varepsilon,\quad
\theta=\frac{1-r}{\omega}.
\end{align*}
Notice $\sigma-\omega=-1-\varepsilon<-1$ and $p'(1-\theta)/\theta<3$. By calculation it can be checked that the latter is equivalent to $r<p'\varepsilon/(3+p')$ in \eqref{EandEC}. Hence
\begin{align*}
\int^\infty_0 (1+t)^{\sigma-\omega} dt
<\infty,\quad
\int_{\vert k\vert\le 1}\vert k \vert^{-p'\frac{1-\theta}{\theta}} dk
<\infty
\end{align*}
by the choice of $r$. This further leads us to
\begin{align*}
&\int_{\vert k\vert \le 1} \Big( (1+t)^{\sigma-\omega} \vert k\vert^{-2\frac{1-\theta}{\theta}} \Vert \hat{f}\Vert^2_{L^2_{v,D}} dt \Big)^{1/2} dk\\
&\le \Big(\int^\infty_0 (1+t)^{\sigma-\omega} dt\Big)^{1/2} \Big(\int_{\vert k\vert\le 1}\vert k \vert^{-p'\frac{1-\theta}{\theta}} dk \Big)^{1/p'} \Vert f\Vert_{L^p_k L^\infty_T L^2_{v,D}}\\
&\le C\Vert f\Vert_{L^p_k L^\infty_T L^2_{v,D}}.
\end{align*}
Using \eqref{ineq:InterpolationE} for $\vert k\vert \ge 1$, we have
\begin{align*}
&\int_{\mathbb{R}^3} \Big(\int^T_0 (1+t)^{\sigma-1} \Vert \hat{f} \Vert_{L^2_v}^2 dt \Big)^{1/2} dk\\
&\le \Big(\int_{\vert k\vert \le 1}+\int_{\vert k\vert \ge 1}\Big) \Big(\int^T_0 (1+t)^{\sigma-1+r} \Vert \hat{f} \Vert_{L^2_{v,D}}^2 dt \Big)^{1/2} dk\\
&\le \int_{\vert k\vert \le 1} \Big(\int^T_0 (\eta^2(1+t)^\sigma \vert k\vert^2+C_\eta^2(1+t)^{\sigma-\omega}\vert k\vert^{-2\frac{1-\theta}{\theta}}) \Vert \hat{f} \Vert_{L^2_{v,D}}^2 dt \Big)^{1/2} dk\\
&+\int_{\vert k\vert \ge 1}\Big(\int^T_0 (\eta^2(1+t)^\sigma +C_\eta^2) \Vert \hat{f} \Vert_{L^2_{v,D}}^2 dt \Big)^{1/2} dk\\
&\le \eta \Vert \vert\nabla_x\vert (1+t)^{\sigma/2} f \Vert_{L^1_{\vert k\vert \le 1} L^2_T L^2_{v,D}} + C_\eta \Vert f \Vert_{L^p_k L^\infty_T L^2_{v,D}} \\
& + \eta \Vert (1+t)^{\sigma/2} f \Vert_{L^1_{\vert k\vert \ge 1} L^2_T L^2_{v,D}} + C_\eta \Vert f\Vert_{L^1_{\vert k\vert \ge 1} L^2_T L^2_{v,D}}.
\end{align*}
The terms multiplied with $\eta$ are absorbed into (recall the argument subsequent to \eqref{ineq: weight high freq}), and
by Lemmas \ref{lem: L^1 soft} and \ref{lem: L^p soft} the others are bounded as follows:
\begin{align*}
&\Vert f\Vert_{L^p_k L^\infty_T L^2_{v,D}}\le \Vert f\Vert_{L^p_k L^\infty_T L^2_v}
\le C\Vert w_{\ell,q} f_0\Vert_{L^p_k L^2_v},\\
&\Vert f\Vert_{L^1_{\vert k\vert \ge 1} L^2_T L^2_{v,D}} \le \Vert w_{\ell,q} f\Vert_{L^1_{\vert k\vert \ge 1} L^2_T L^2_{v,D}}\le C\Vert w_{\ell,q} f_0 \Vert_{L^1_k L^2_v}.
\end{align*}
Hence \eqref{time-remainder-soft3} is handled.

\medskip 
\noindent\textbf{Step 2:
Estimates on $E^c$}.
First, we consider  \eqref{time-remainder-soft1}
\begin{align*}
\int_{\vert k\vert \le 1} \Big( \int^T_0 \int_{E^c} (1+t)^{\sigma-1} \vert w_{\ell,q} (\mathbf{I}-\mathbf{P}) \hat{f}\vert^2 dv dt \Big)^{1/2} dk.
\end{align*}
For this term, we use the interpolation
\begin{align*}
(1+t)^{\sigma -1} 
&= (1+t)^{(1-\theta)\sigma}\langle v\rangle^{(1-\theta)(\gamma+2s)}\cdot (1+t)^{\theta(\sigma-\frac{r_1}{r_1-1})}\langle v\rangle^{\theta\frac{\vert \gamma+2s\vert}{r_1-1}}\\
& \le \eta^2 (1+t)^{\sigma} \langle v\rangle^{\gamma+2s} +  C_\eta^2 (1+t)^{\sigma-\frac{r_1}{r_1-1}} \langle v\rangle^{\frac{\vert \gamma+2s\vert}{r_1-1}}
\end{align*}
for some $1<r_1<\infty$ chosen later.
Here $\theta=1-1/r_1\in (0,1)$.
We set $2\alpha=1/(r_1-1)$ for brevity.
The first term is controlled by
\begin{align*}
&\int_{\vert k\vert \le 1} \Big( \int^T_0 \int_{\mathbb{R}^3} (1+t)^{\sigma} \langle v\rangle^{\gamma+2s} \vert w_{\ell,q} (\mathbf{I}-\mathbf{P}) \hat{f}\vert^2 dv dt \Big)^{1/2} dk\\
&\le \int_{\vert k\vert \le 1} \Big( \int^T_0 (1+t)^{\sigma} \Vert w_{\ell,q} (\mathbf{I}-\mathbf{P}) \hat{f}\Vert^2_{L^2_{v,D}} dt \Big)^{1/2} dk
=\Vert w_{\ell,q} (\mathbf{I}-\mathbf{P}) f\Vert_{L^1_{\vert k\vert\le 1} L^2_T L^2_{v,D}},
\end{align*}
and the second one is bounded by
\begin{align}
&\int_{\vert k\vert \le 1} \Big( \int^T_0 \int_{E^c} (1+t)^{\sigma-\frac{r_1}{r_1-1}} \langle v\rangle^{2\vert \gamma+2s\vert\alpha} \vert w_{\ell,q} (\mathbf{I}-\mathbf{P}) \hat{f}\vert^2 dv dt \Big)^{1/2} dk \notag \\
&\le \int_{\vert k\vert \le 1} \Big( \int^T_0 (1+t)^{\sigma-\frac{r_1}{r_1-1}} \Vert w_{\ell+\alpha,q} (\mathbf{I}-\mathbf{P}) \hat{f}\Vert_{L^2_v}^2  dt \Big)^{1/2} dk \notag\\
&\le C \Vert w_{\ell+\alpha,q} (\mathbf{I}-\mathbf{P}) f\Vert_{L^p_{\vert k\vert\le 1} L^\infty_T L^2_v}\Big(\int^\infty_0 (1+t)^{\sigma-\frac{r_1}{r_1-1}} dt \Big)^{1/2}  \notag \\
&\le C \Vert w_{\ell+\alpha,q} f_0\Vert_{L^p_k L^2_v} \label{ineq: j_1}
\end{align}
by Lemma \ref{lem: L^p soft},
provided
\begin{align*}
\sigma-\frac{r_1}{r_1-1}<-1
\Leftrightarrow \sigma< \frac{1}{r_1-1}=2\alpha.
\end{align*}
Therefore the infimum of $\alpha$ in \eqref{ineq: j_1} is $\sigma/2$. We take $r_1=1+(\sigma+\varepsilon')^{-1}$, $\varepsilon'>0$ so that $\alpha$ is close to $\sigma/2$ as we wish.
Since $j>\sigma/(2r)>\sigma/2$ under \eqref{ineq:SoftIndexCondition}, the bound is replaced with $C\Vert w_{\ell+j,q} f_0\Vert_{L^p_k L^2_v}$ if $\varepsilon'>0$ is small enough.

Next, using $1\le (1+t)^{-r} \langle v\rangle^{\vert \gamma+2s\vert}$ on $E^c$, for \eqref{time-remainder-soft2} we have
\begin{align}
&\int_{\vert k\vert\ge 1} \Big(\int^T_0 \int_{E^c} (1+t)^{\sigma-1} \vert w_{\ell,q} \hat{f}\vert^2 dvdt\Big)^{1/2} dk \notag \\
&\le \int_{\vert k\vert\ge 1} \Big(\int^T_0 \int_{E^c} (1+t)^{\sigma-1-2r j} \vert w_{\ell+j,q} \hat{f}\vert^2 dvdt\Big)^{1/2} dk \notag \\
& \le C \Vert w_{\ell+j,q}  f\Vert_{L^1_{\vert k\vert \ge 1} L^\infty_T L^2_v} \Big(\int^\infty_0 (1+t)^{\sigma-1-2r j} dt\Big)^{1/2} \notag \\
& \le C \Vert w_{\ell+j,q}  f_0\Vert_{L^1_k L^2_v} \notag
\end{align}
by Lemma \ref{lem: L^1 soft} because of 
$\sigma-1-2rj<-1$ under the assumption \eqref{ineq:SoftIndexCondition}.

We divide \eqref{time-remainder-soft3} into the two parts
\begin{align*}
\Big(\int_{\vert k\vert \ge 1} + \int_{\vert k\vert \le 1}\Big)
\Big(\int^T_0 \int_{E^c}(1+t)^{\sigma-1} \vert \hat{f}\vert^2 dv dt\Big)^{1/2} dk.
\end{align*}
For the part $\vert k\vert \ge 1$, we apply the same technique used in \cite{DLSS}. That is, on $E^c$ we have
\begin{align*}
w^{-2}_{\ell,q} \le (1+t)^{-2r\ell} e^{-\frac{q}{2}(1+t)^{\frac{r}{\vert \gamma+2s\vert}}}. 
\end{align*}
Plugging this into the integrand gives
\begin{align*}
&\int_{\vert k\vert\ge 1} \Big(\int^T_0 \int_{E^c}(1+t)^{\sigma-1} \vert \hat{f}\vert^2 dv dt\Big)^{1/2} dk\\
&\le \int_{\vert k\vert\ge 1} \Big(\int^T_0 (1+t)^{\sigma-1-2r\ell} e^{-\frac{q}{2}(1+t)^{\frac{r}{\vert \gamma+2s\vert}}} \Vert w_{\ell,q} \hat{f}\Vert_{L^2_v}^2 dt\Big)^{1/2} dk\\
&\le C \Vert w_{\ell,q} f\Vert_{L^1_{\vert k\vert\ge1} L^\infty_T L^2_v}  \Big(\int^\infty_0 (1+t)^{\sigma-1-2r\ell} e^{-\frac{q}{2}(1+t)^{\frac{r}{\vert \gamma+2s\vert}}} dt\Big)^{1/2}\\
&\le C\Vert w_{\ell,q} f_0\Vert_{L^1_k L^2_v}
\end{align*}
for any $r>0$, $q>0$, and $\ell\ge 0$. 
If $q=0$ (polynomial weight case), recall \eqref{ineq:SoftIndexCondition} so that $\sigma-1-2r\ell<-1$.

Finally we consider the part $\vert k\vert \le 1$ of \eqref{time-remainder-soft3}. It holds that
\begin{align}
&\int_{\vert k\vert \le 1} \Big( \int^T_0 \int_{E^c} (1+t)^{\sigma-1} \vert \hat{f}\vert^2 dvdt\Big)^{1/2} dk \notag\\
&\le \sqrt{2}\int_{\vert k\vert\le 1} \Big( \int^T_0 \int_{E^c} (1+t)^{\sigma-1} \vert (\mathbf{I}-\mathbf{P}) \hat{f}\vert^2 dvdt\Big)^{1/2} dk
\label{finalterm1}\\
&+\sqrt{2}\int_{\vert k\vert\le 1} \Big( \int^T_0 \int_{E^c} (1+t)^{\sigma-1} \vert \mathbf{P} \hat{f}\vert^2 dvdt\Big)^{1/2} dk
\label{finalterm2}.
\end{align}
Since $1\le w_{\ell,q}$, the estimate of \eqref{finalterm1} can be carried out as \eqref{time-remainder-soft1} was done. For \eqref{finalterm2} we use the interpolation
\begin{align*}
(1+t)^{\sigma-1} \le \eta^2 (1+t)^\sigma \vert k\vert^2 + C_\eta^2 (1+t)^{\sigma-\frac{r_2}{r_2-1}} \vert k\vert^{-\frac{2}{r_2-1}},\quad
r_2>{\frac{4p-3}{3p-3}}(>1).
\end{align*}
Notice the condition on $r_2$ is equivalent to $p'/(r_2-1)<3$.
Recalling $\vert k\vert^2\le \frac{2\vert k\vert^2}{1+\vert k\vert^2}$, we get 
\begin{align*}
\int_{\vert k\vert\le 1} \Big( \int^T_0 \int_{E^c} (1+t)^\sigma \vert k\vert^2 \vert \mathbf{P}\hat{f}\vert^2 dvdt\Big)^{1/2} dk\le C \Big\Vert \frac{\vert \nabla_x \vert}{\langle \nabla x\rangle} (1+t)^{\sigma/2} (a,b,c) \Big\Vert_{L^1_{\vert k\vert \le 1} L^2_T},
\end{align*}
which is absorbed. Also, by Lemma \ref{lem: L^p soft}, for sufficiently large $j'>0$ it holds
\begin{align*}
& \int_{\vert k\vert \le 1} \Big(\int^T_0 \int_{E^c} (1+t)^{\sigma-\frac{r_2}{r_2-1}} \vert k\vert^{-\frac{2}{r_2-1}} \vert \mathbf{P}\hat{f}\vert^2 dvdt \Big)^{1/2} dk\\
&\le \int_{\vert k\vert \le 1} \Big(\int^T_0 \int_{E^c} (1+t)^{\sigma-\frac{r_2}{r_2-1}-2r j'} \vert k\vert^{-\frac{2}{r_2-1}} \vert w_{j',0} \mathbf{P}\hat{f}\vert^2 dvdt \Big)^{1/2} dk\\
&\le  \Big(\int_{\vert k\vert\le 1} \vert k\vert^{-\frac{p'}{r_2-1}} dk \Big)^{1/p'} \Big(\int^T_0 (1+t)^{\sigma-\frac{r_2}{r_2-1}-2r j'} dt \Big)^{1/2} \Vert w_{j',0} \mathbf{P}f\Vert_{L^p_{\vert k \vert\le 1} L^\infty_T L^2_v}\\
&\le C \Vert f \Vert_{L^p_{\vert k \vert\le 1} L^\infty_T L^2_v}\\
&\le C \Vert w_{\ell+j,q} f_0 \Vert_{L^p_k L^2_v}.
\end{align*}
In the penultimate inequality, we have used the fact that for any $j'>0$, $w_{j',0}\mathbf{P}$ is a bounded operator over $L^2_v$.

\medskip 
\noindent\textbf{Step 3}: We substitute all the estimates of \eqref{time-remainder-soft1}, \eqref{time-remainder-soft2}, \eqref{time-remainder-soft3} obtained in Steps 1 and 2 into \eqref{ineq: time weight I-P low freq}-\eqref{ineq: time weight no velocity}.
Then the rest is quite similar to the proof of Lemma \ref{lem: L^1 soft}.
Indeed, we take an appropriate linear combination of the resultant inequalities, then estimate the nonlinear terms by the macro-micro decomposition. 
Using Lemma \ref{lem: L^1 soft}, if $\Vert w_{\ell, q} f_0 \Vert_{L^1_k L^2_v} +\delta$ is sufficiently small, the nonlinear terms are absorbed into the left hand side. 
This completes the proof.
\end{proof}

\begin{proof}[Proof of Theorem \ref{thm: MainSoft}]
Now, thanks to Lemma \ref{lem: time-weighted soft}, \eqref{apaweak} is verified for the soft potential case provided that
$\Vert w_{\ell+j,q} f_0 \Vert_{L^1_k L^2_v}
+  \Vert w_{\ell+j,q} f_0\Vert_{L^p_k L^2_v}$
is sufficiently small. Furthermore, \eqref{thm.sc.c1} follows from Lemma \ref{lem: time-weighted soft} by noting 
\begin{align}
\mathscr{E}(T)+\mathscr{D}(T) &\simeq  \Vert (1+t)^{\sigma/2} w_{\ell,q} f\Vert_{L^1_k L^\infty_T L^2_v}
+\Vert (1+t)^{\sigma/2} w_{\ell,q} (\mathbf{I}-\mathbf{P})  f \Vert_{L^1_k L^2_T L^2_{v,D}}\notag\\
&\quad+\Big\Vert \frac{\vert \nabla_x \vert}{\langle \nabla_x \rangle} (1+t)^{\sigma/2} (a,b,c) \Big\Vert_{L^1_k L^2_T}.\notag
\end{align}
Similarly, \eqref{thm.sc.c2} follows from Lemma \ref{lem: L^p soft} by noting
\begin{align}
  \mathcal{E}_p(T) + \mathcal{D}_p(T) &\simeq  \Vert w_{\ell,q} f\Vert_{L^p_k L^\infty_T L^2_v}
+\Vert  w_{\ell,q} (\mathbf{I}-\mathbf{P})  f \Vert_{L^p_k L^2_T L^2_{v,D}}\notag\\
&\quad +\Big\Vert \frac{\vert \nabla_x \vert}{\langle \nabla_x \rangle}  (a,b,c) \Big\Vert_{L^p_k L^2_T}.\notag
\end{align}
We then have proved the global existence and large-time behavior of solutions for the case of soft potentials. 
\end{proof}

\section{Appendix}
Here we give some auxiliary statements and proofs that have been postponed in the previous sections.

On local existence, it suffices to repeat the argument in \cite{DLSS}, which was motivated by \cite{AMUXY-2011-CMP} and \cite{MS-JDE-2016}. The key idea is that, since $L^1_k \cap L^p_k\subset L^1_k \cap L^\infty_k \subset L^2_k \simeq L^2_x$, we can work on $L^\infty_{T_0} L^2_{x,v}$ to construct a solution to an approximate equation via the Hahn-Banach extension theorem as in the same way of \cite{DLSS}. Compared to the argument in \cite{DLSS}, our case is rather simple because the estimates of derivatives is not necessary and there is no boundary. We do not give full proof.

\subsection{Proof of local existence}
In this subsection, we will show the following claim of local existence for the hard potential case and $p=\infty$. 
Let $X=L^1_k\cap L^\infty_k$. Under the same assumption as in Theorem \ref{thm: MainHard}, there exist $\epsilon>0$, $T_0>0$, and $C>0$ such that if $F_0(x,v)=\mu(v)+\mu(v)^{1/2} f_0 (x,v)\ge 0$ and
\begin{align*}
\Vert f\Vert_{X L^\infty_{T_0} L^2_v} + \Vert f\Vert_{X L^2_{T_0} L^2_{v,D}} \le \epsilon, 
\end{align*}
then \eqref{BE} admits a unique local solution
\begin{align*}
f(t,x,v),\ 0\le t\le T_0,\ x,\ v\in\mathbb{R}^3.
\end{align*}
This solution belongs to
\begin{align*}
L^1_k L^\infty_{T_0} L^2_v \cap L^\infty_k L^\infty_{T_0} L^2_v \cap L^1_k L^2_{T_0} L^2_{v,D} \cap L^\infty_k L^2_{T_0} L^2_{v,D}
\end{align*}
and satisfies
\begin{align}\label{ineq:EnergyLocalSol}
\Vert f\Vert_{X L^\infty_{T_0} L^2_v}+ \Vert f\Vert_{X L^2_{T_0} L^2_{v,D}}\le C \Vert f_0 \Vert_{X L^2_v}.
\end{align}
We remark that the same proof below applies to the soft potential case, where we have to define a solution space according to \eqref{eq:SoftL^1Energy}, \eqref{eq:SoftL^1Diss}, \eqref{eq:SoftL^pEnergy}, and \eqref{eq:SoftL^pDiss}.

We mainly follow \cite{DLSS}, while we also remark that the idea in \cite{DLSS} came from \cite{AMUXY-2011-CMP} and \cite{MS-JDE-2016}. 

First, we consider the linearized equation
\begin{align}\label{eq:linearised}
\begin{cases}
\partial_t f +v \cdot \nabla_x f + L_1  f  -\Gamma(g , f) = -L_2 g,\\
f(0,x,v)=f_0(x,v).
\end{cases}
\end{align}
Recall $L_1f=-\mu^{1/2}Q(\mu,\mu^{1/2} f)$ and $L_2f=-\mu^{1/2}Q(\mu^{1/2} f,\mu)$. They satisfy~\cite{AMUXY-2012-JFA}
\begin{align}
&(L_1f, f)_{L^2_v} \ge \delta_0 \Vert f\Vert_{L^2_{v,D}}-C \Vert f\Vert_{L^2_v}, \label{ineq:L1} \\
& \vert (L_2 f, g)_{L^2_v}\vert\le C \Vert \mu^{1/10^3} f \Vert_{L^2_v} \Vert \mu^{1/10^3} g \Vert_{L^2_v}\le C \Vert f \Vert_{L^2_{v,D}} \Vert g \Vert_{L^2_{v,D}}.
\label{ineq:L2}
\end{align} 
To prove that \eqref{eq:linearised} has a unique weak solution, we have to mollify the functions. Hence, applying the standard Friedrichs mollifier $\{ \chi_\varepsilon(x)\}_{0<\varepsilon<1}$, we solve
\begin{align}\label{eq: mollified}
\begin{cases}
\partial_t f_\varepsilon +v \cdot \nabla_x f_\varepsilon + L_1  f_\varepsilon  -\Gamma(g_\varepsilon , f_\varepsilon) = -L_2 g_\varepsilon,\\
f_\varepsilon(0,x,v)=f_{0, \varepsilon}(x,v),
\end{cases}
\end{align}
under the assumption
\begin{align*}
&g_\varepsilon \in XL^\infty_{T_0} L^2_v\cap X L^2_{T_0} L^2_{v,D},\quad f_{0,\varepsilon}\in XL^2_v,\\
&\Vert g_\varepsilon\Vert_{XL^\infty_{T_0}L^2_v} + \Vert g_\varepsilon\Vert_{XL^2_{T_0}L^2_{v,D}} \le \epsilon,
\end{align*}
where
\begin{align*}
g_\varepsilon= g*_x\chi_\varepsilon,\ f_{0,\varepsilon}=f_0 *_x \chi_\varepsilon.
\end{align*}

Since $L^1_k \cap L^\infty_k \subset L^2_k \simeq L^2_x$, we work on $L^\infty_{T_0} L^2_{x,v}$ to construct a weak solution via the Hahn-Banach theorem. 
Let us define 
\begin{align*}
\mathscr{Q}=-\partial_t + \Big(v\cdot \nabla_x +L_1 -\Gamma(g_\varepsilon,\cdot)\Big)^*,
\end{align*}
where $*$ denotes the conjugate operator with respect to the standard inner product over $L^2_{x,v}$. This operator is defined over
\begin{align*}
\mathbb{W}_1=\{ f\in H^1(0,T_0; \mathcal{S}(\mathbb{R}^6_{x,v}))\ |\ f(T_0,x,v)\equiv 0\},
\end{align*}
where $\mathcal{S}(\mathbb{R}^d)$ denotes the set of rapidly decreasing functions on $\mathbb{R}^d$.  Then it holds
\begin{align*}
\mathrm{Re} (\mathscr{Q} f_\varepsilon, f_\varepsilon)_{L^2_{x,v}} =-\frac{1}{2}\frac{d}{dt} \Vert f_\varepsilon\Vert_{L^2_{x,v}}^2 +(L_1 f_\varepsilon, f_\varepsilon)_{L^2_{x,v}} -\mathrm{Re}(\Gamma(g_\varepsilon, f_\varepsilon), f_\varepsilon)_{L^2_{x,v}}.
\end{align*}
By the trilinear estimate \eqref{ineq:Trilinear} it holds
\begin{align*}
\vert (\Gamma(g_\varepsilon, f_\varepsilon), f_\varepsilon)_{L^2_{x,v}} \vert
\le C \int_{\mathbb{R}^3_x} \Vert g_\varepsilon \Vert_{L^2_v} \Vert f_\varepsilon \Vert_{L^2_{v,D}}^2 dx 
\le C \Vert g_\varepsilon \Vert_{L^\infty_x L^2_v} \Vert f_\varepsilon\Vert_{L^2_x L^2_{v,D}}^2.
\end{align*}
Therefore if $\epsilon$ is sufficiently small, using \eqref{ineq:L1} we have
\begin{align}\label{ineq: local1}
&\frac{1}{2}\Vert f_\varepsilon (t) \Vert_{L^2_{x,v}} +(\delta_0 - C\epsilon) \int^{T_0}_t \Vert f_\varepsilon(t)\Vert_{L^2_x L^2_{v,D}} d\tau \notag\\ 
&\le \int^{T_0}_t \vert (\mathscr{Q}f_\varepsilon,f_\varepsilon)_{L^2_{x,v}}\vert dt+C \int^{T_0}_t \Vert f_\varepsilon \Vert_{L^2_x L^2_{v,D}} d\tau \notag\\
& \le (CT_0+\eta) \Vert f_\varepsilon \Vert_{L^\infty_{T_0} L^2_{x,v}}^2 +C\Vert \mathscr{Q} f_\varepsilon \Vert_{L^1_{T_0} L^2_{x,v}}^2.
\end{align}
This implies that $\mathscr{Q}$ is injective on $\mathbb{W}_1$.

Now we can define a map $\mathcal{M}$ over $\mathbb{W}_2$ as follows:
\begin{align*}
&\mathbb{W}_2 = \{ w=\mathscr{Q} f\ |\ f\in \mathbb{W}_1\} \subset L^1_{T_0} L^2_{x,v},\\
&\mathcal{M}: \mathbb{W}_2 \rightarrow \mathbb{C},\
w_\varepsilon=\mathscr{Q}f_\varepsilon\mapsto (f_\varepsilon(0), f_{0,\varepsilon})_{L^2_{x,v}}-\int^{T_0}_0 (L_2 g_\varepsilon, f_\varepsilon)_{L^2_{x,v}} dt.
\end{align*}
This map is well-defined, because $\mathscr{Q}: \mathbb{W}_1\rightarrow \mathbb{W}_2$ is a bijection. Also, $\mathcal{M}$ is bounded by
\begin{align*}
\vert \mathcal{M}(w_\varepsilon)\vert 
&\le  \Vert f_\varepsilon(0)\Vert_{L^2_{x,v}} \Vert f_{0,\varepsilon}\Vert_{L^2_{x,v}} +C\Vert g_\varepsilon \Vert_{L^\infty_{T_0} L^2_{x,v}}  \Vert  f_\varepsilon \Vert_{L^1_{T_0} L^2_{x,v}}\\
&\le C(T_0,\epsilon) ( \Vert f_{0,\varepsilon} \Vert_{L^2_{x,v}} +\Vert g_\varepsilon\Vert_{L^\infty_{T_0} L^2_{x,v}}) \Vert \mathscr{Q} f_\varepsilon\Vert_{L^\infty_{T_0} L^2_{x,v}}.
\end{align*}
By the Hahn-Banach theorem, $\mathcal{M}$ can be extended to a bounded linear functional on $L^1_{T_0}L^2_{x,v}$ and there exists $f_\varepsilon\in L^\infty_{T_0} L^2_{x,v}$ such that it holds
\begin{align*}
\mathcal{M}(w_\varepsilon)=\int^{T_0}_0 (f_\varepsilon (t), w_\varepsilon (t))_{L^2_{x,v}} dt,\ w_\varepsilon \in L^1_{T_0}L^2_{x,v},\\
\Vert f_\varepsilon\Vert_{L^\infty_{T_0} L^2_{x,v}} \le C(T_0,\epsilon) ( \Vert f_{0,\varepsilon} \Vert_{L^2_{x,v}} +\Vert g_\varepsilon\Vert_{L^\infty_{T_0} L^2_{x,v}}).
\end{align*}
This is a weak solution to \eqref{eq: mollified}. Indeed,
\begin{align*}
\mathcal{M}(\mathscr{Q}f_\varepsilon)=\int^{T_0}_0 (f_\varepsilon (t), \mathscr{Q}f_\varepsilon)_{L^2_{x,v}} dt=(f_\varepsilon(0), f_{0,\varepsilon})_{L^2_{x,v}}-\int^{T_0}_0 (L_2 g_\varepsilon, f_\varepsilon)_{L^2_{x,v}} dt.
\end{align*}
It is obvious to show $f_\varepsilon \in L^2_{T_0} L^2_x L^2_{v,D}$ by \eqref{ineq: local1}.

Next, we verify that
\begin{align}\label{convergence}
f_\varepsilon \rightarrow \exists f \in XL^\infty_{T_0}L^2_v \cap XL^2_{T_0} L^2_{v,D}
\end{align}
as $\varepsilon \rightarrow 0$. By the energy estimate, we deduce
\begin{align*}
&\Vert \hat{f}_\varepsilon(t)\Vert_{L^2_v} +\sqrt{\delta_0} \Vert \hat{f}_\varepsilon(t)\Vert_{L^2_{T_0}L^2_{v,D}} \\
&\le \Vert \hat{f}_{0,\varepsilon}\Vert_{L^2_v}
+C \Vert \hat{f}_\varepsilon(t)\Vert_{L^2_{T_0}L^2_v}\\
&+\Big(\int^{T_0}_0 \vert (L_2 \hat{f}_\varepsilon, \hat{g}_\varepsilon)\vert dt\Big)^{1/2}
+\Big(\int^{T_0}_0 \vert (\hat{\Gamma}(\hat{g}_\varepsilon, \hat{f}_\varepsilon), \hat{f}_\varepsilon)\vert dt\Big)^{1/2}.
\end{align*}
Since we have
\begin{align*}
\Big\Vert \Big(\int^{T_0}_0 \vert (L_2 \hat{g}_\varepsilon, \hat{f}_\varepsilon)\vert dt\Big)^{1/2} \Big\Vert_X
&\le \eta \Vert f_\varepsilon\Vert_{X L^2_{T_0} L^2_{v,D}} +C_\eta \Vert g_\varepsilon\Vert_{X L^2_{T_0} L^2_{v,D}},\\ 
\Big\Vert \Big(\int^{T_0}_0 \vert (\hat{\Gamma}(\hat{g}_\varepsilon, \hat{f}_\varepsilon), \hat{f}_\varepsilon)\vert dt\Big)^{1/2} \Big\Vert_X
&\le \eta \Vert f_\varepsilon \Vert_{X L^2_{T_0} L^2_{v,D}}\\
&+C_\eta  \Vert g_\varepsilon \Vert_{L^1_k L^2_{T_0}  L^2_{v,D}} \Vert f_\varepsilon \Vert_{X L^2_{T_0} L^2_{v,D}}
\end{align*}
by \eqref{ineq:Trilinear} and \eqref{ineq:L2},
the substitution of these estimates yields
\begin{align*}
&\Vert f_\varepsilon \Vert_{XL^\infty_{T_0} L^2_v}
+\sqrt{\delta_0} \Vert f_\varepsilon \Vert_{XL^2_{T_0} L^2_{v,D}}\\
&\le \Vert f_{0,\varepsilon} \Vert_{X L^2_v}
+\sqrt{T_0} \Vert f_\varepsilon \Vert_{X L^\infty_{T_0} L^2_v}
+\eta \Vert f_\varepsilon \Vert_{X L^2_{T_0} L^2_{v,D}} 
+C_\eta \Vert g_\varepsilon \Vert_{X L^2_{T_0} L^2_{v,D}}\\
&+C_\eta \Vert g_\varepsilon \Vert_{L^1_k L^2_{T_0}  L^2_{v,D}} \Vert f_\varepsilon \Vert_{X L^2_{T_0} L^2_{v,D}}.
\end{align*}
Hence, if $T_0>0$ and $\epsilon>0$ are sufficiently small, we have the estimate
\begin{align*}
\Vert f_\varepsilon \Vert_{XL^\infty_{T_0} L^2_v}
+\lambda \Vert f_\varepsilon \Vert_{XL^2_{T_0} L^2_{v,D}}
\le C \Vert f_{0} \Vert_{X L^2_v}
+ C \Vert g\Vert_{X L^2_{T_0} L^2_{v,D}}
\end{align*}
uniform in $0<\varepsilon<1$.
Notice we have used $\Vert f_\varepsilon \Vert_X \le \Vert f \Vert_X$ uniform in $\varepsilon$. 
This gives \eqref{convergence}.

Finally, we inductively define $\{f_n\}_{n\in \mathbb{N}}$ so as $f_n$ is a weak solution to
\begin{align*}
\begin{cases}
\partial_t f_n +v \cdot \nabla_x f_n + L_1  f_n  -\Gamma(f_{n-1} , f_n) = -L_2 f_{n-1},\\
f_n(0,x,v)=f_0(x,v).
\end{cases}
\end{align*}
By the aforementioned argument, $f_n\rightarrow \exists f$ in $XL^\infty_{T_0} L^2_v \cap XL^2_{T_0} L^2_{v,D}$ if $f_0$ is sufficiently small in $X$. This completes the proof of local existence.

\subsection{Proof of non-negativity}
In this subsection, we will prove the non-negativity of the local solution we have obtained in the previous subsection.

The method of the proof is essentially the same as that of \cite[Proposition 5.1]{AMUXY-ARMA-2011}, however, as it is noted in \cite{MS-JDE-2016}, we use \cite[Corollary 6.20]{MS-JDE-2016} instead of \cite[Lemma 4.9]{AMUXY-ARMA-2011} to handle $L^2_{v,D}$-norms so that we can avoid using the Sobolev norms.  For self-containedness, we reproduce it.  

If $\{f_n\}$ is the sequence of approximate solutions in the previous proof, and if $F_n = \mu + \mu^{1/2}f_n$, then 
$\{F_n\}$ is constructed successively by the following linear Cauchy problem
\begin{align*}
\begin{cases}
\partial_t F_{n+1} + v\,\cdot\,\nabla_x
F_{n+1} =Q (F_n, F_{n+1}), \\ 
F_{n+1}|_{t=0} = F_0 =\mu + \mu^{1/2}
 f_0\geq 0\, ,\ n\in\mathbb{N}.
\end{cases}
\end{align*}
Hence the non-negativity of the solution to the original Cauchy problem \eqref{eq: Boltzmann} comes  from the following induction argument: 
Suppose  that
\begin{align}\label{4.4.1+100}
F_n = \mu + \mu^{1/2} f_n \geq 0
\end{align}
for some $n\in\mathbb{N}$. We will show that then  (\ref{4.4.1+100}) is true
for $n+1$.

If we put $\tilde F_n = \mu^{-1/2}F_n = \mu^{1/2} + f_n$ then $\tilde F_n$ satisfies 
\begin{align}\label{4.4.3-bis}
\begin{cases}
\partial_t \tilde F_{n+1} + v\,\cdot\,\nabla_x
\tilde F_{n+1} =\Gamma (\tilde F_n, \tilde F_{n+1}), \\
\tilde F_{n+1}|_{t=0} =\tilde F_0 = \mu^{1/2} + 
 f_0\geq 0\, ,\ n\in\mathbb{N}.
\end{cases}
\end{align}
Since $L^1_k L^2_T L^2_{v,D} \subset L^\infty_x L^2_T L^2_{v,D} \subset L^2_T L^\infty_x L^2_{v,D},$
it holds $\int_0^T \Vert \tilde{F}_n\Vert^2_{L^\infty_xL^2_{v,D}} dt < \infty$,
and hence, if $\tilde F_n^{\pm} = \pm \max(\pm \tilde F_n, 0)$ then both
$\int_0^T \Vert \tilde{F}_n^\pm\Vert^2_{L^\infty_xL^2_{v,D}} dt$ are also finite
by means of the same argument as in the proof of \cite[Proposition 5.1]{AMUXY-ARMA-2011}. 
Take the convex function $\beta (s) = \frac 1 2 (s^- )^2= \frac 1 2
s\,(s^- )  $ with $s^-=\min\{s, 0\}$.  Let $\varphi(v,x) = (1+|v|^2 +|x|^2)^{\alpha/2}$ with $\alpha >3/2$,
and notice that
\begin{align*}
\beta_s (\tilde F_{n+1}) \varphi(v,x)^{-1}&: =\left(\frac{d}{ds}\,\,\beta
\right)(\tilde F_{n+1}) \varphi(v,x)^{-1}\\
&={\tilde F^{ -}_{n+1}} \varphi(v,x)^{-1}\in
L^\infty_T L^2_{x,v}.
\end{align*}

Multiply the first equation of \eqref{4.4.3-bis}
by $\beta_s (\tilde F_{n+1})\varphi(v,x)^{-2}$ $
= \tilde F^{-}_{n+1}\varphi(v,x)^{-2}$ and integrate over
$[0,t] \times \mathbb{R}^6_{x,v}$, ($t \in (0,T]$). Then, since $\beta(F_{n+1}(0)) =  (\tilde F_{0}^-)^2/2 =0$, we have
\begin{align*}
&\int_{\mathbb{R}^6} \beta ( \tilde F_{n+1}(t)) \varphi(v,x)^{-2}dxdv
\\
&\qquad =\int_0^t \int_{\mathbb{R}^6}
\Gamma(\tilde F_n(\tau),\, \tilde F_{n+1}(\tau) )\,\, \beta_s(\tilde F_{n+1}(\tau)) \varphi(v,x)^{-2} \,\,dxdvd\tau \\
&\qquad\qquad-\int_0^t \int_{\mathbb{R}^6} { v\,\cdot\, \nabla_x \enskip  (  \beta
(\tilde F_{n+1}(\tau)) \varphi(v,x)^{-2}) }dxdv d\tau\\
&\qquad\qquad + \int_0^t \int_{\mathbb{R}^6} {\big (\varphi(v,x)^{2}\,\,v\, \cdot\,
\nabla_x \,\varphi(v,x)^{-2} \big) }\enskip  \beta (\tilde F_{n+1}(\tau))\varphi(v,x)^{-2}dxdvd\tau, 
\end{align*}
where the first term on the right-hand side is well-defined because 
\begin{align*}
\int_0^T \Vert \tilde{F}_{n+1}\Vert^2_{L^\infty_xL^2_{v,D}} dt,\quad
\int_0^T \Vert \tilde{F}_{n+1}^-\Vert^2_{L^\infty_xL^2_{v,D}} dt < \infty.
\end{align*}
Since the second term vanishes
and $|v\, \cdot\, \nabla_x\, \varphi(v,x)^{-2} | \leq C \varphi(v,x)^{-2}$,
we obtain
\begin{align*}
&\int_{\mathbb{R}^6} \beta ( \tilde F_{n+1}(t)) \varphi(v,x)^{-2}dxdv
\\
&\qquad \le \int_0^t \Big(\int_{\mathbb{R}^6}
\Gamma(\tilde F_n(\tau),\, \tilde F_{n+1}(\tau) )\,\, \beta_s(\tilde F_{n+1}(\tau)) \varphi(v,x)^{-2} \,\,dxdv\Big)d\tau \\
&\qquad\qquad \qquad \qquad \qquad +C \int_0^t 
\int_{\mathbb{R}^6} \beta ( \tilde F_{n+1}(\tau)) \varphi(v,x)^{-2}dxdvd\tau.
\end{align*}
The integrand $(\cdot)$ of the  first term on the right-hand side is equal to
\begin{align*}
&\int_{\mathbb{R}^6}  \Gamma(\tilde F_n,  \tilde F_{n+1}^- ) \tilde F_{n+1}^- \varphi(v,x)^{-2} dxdv\\
&\qquad\qquad\qquad + \int   B \, \mu^{1/2}(u) \tilde F_n(u') \tilde F_{n+1}^+(v') \tilde 
 F_{n+1}^- \varphi(v,x)^{-2}dvdu d\sigma dx \\
 &=A_1 + A_2.
\end{align*}
From the induction hypothesis, the second term $A_2$ is non-positive.

On the other hand, we have
\begin{align*}
&A_1 = \int( \Gamma  (\tilde F_n, \varphi(v,x)^{-1}\tilde  F_{n+1}^-),  \varphi(v,x)^{-1} \tilde F_{n+1}^-)_{L^2(\mathbb{R}_v^3)} dx + R
\\
&= - \int (L_1(\varphi(v,x)^{-1}\tilde  F_{n+1}^-), \varphi(v,x)^{-1}\tilde  F_{n+1}^-)_{L^2(\mathbb{R}_v^3)} dx\\ 
 & + \int( \Gamma  ( f^n, \varphi(v,x)^{-1}\tilde  F_{n+1}^-),  \varphi(v,x)^{-1} F_{n+1}^-)_{L^2(\mathbb{R}_v^3)} dx +R,
\end{align*}
where $R$ is a remainder term. It follows from \cite[Corollary 6.20]{MS-JDE-2016}
that
\begin{align*}
\int_0^t |R| d\tau \leq&  \eta \int_0^t \int_{\mathbb{R}_x^3}
\Vert \varphi(v,x)^{-1}\tilde  F_{n+1}^-(\tau) \Vert_{L^2_{v,D}}^2 dx d\tau  \\
&+  C_\eta \int_0^t \int_{\mathbb{R}^6} \beta ( \tilde F_{n+1}(\tau)) \varphi(v,x)^{-2}dxdvd\tau.
\end{align*}
Moreover, by means of \eqref{ineq:L1} and \cite[Theorem 1.2]{AMUXY-KRM-2013} with \eqref{ineq:EnergyLocalSol}, we obtain
\begin{align*}
\int_0^t A_1 d\tau 
&\le -(\delta_0 - C(\eta + \epsilon))
\int_0^t \int_{\mathbb{R}_x^3}
\Vert \varphi(v,x)^{-1}\tilde  F_{-}^{n+1}(\tau) \Vert_{L^2_{v,D}}^2 dx  d\tau \\
&+  C_\eta \int_0^t \int_{\mathbb{R}^6} \beta ( \tilde F_{n+1}(\tau)) \varphi(v,x)^{-2}dxdvd\tau.
\end{align*}
Finally, we get 
\begin{align*}
\int_{\mathbb{R}^6} \beta ( \tilde F_{n+1}(t)) \varphi(v,x)^{-2}dxdv
\lesssim  \int_0^t 
\int_{\mathbb{R}^6} \beta ( \tilde F_{n+1}(\tau)) \varphi(v,x)^{-2}dxdvd\tau,
\end{align*}
which implies that 
$\tilde F_{n+1}(t,x,v) \ge 0$ for $(t,x,v) \in [0,T]\times \mathbb{R}^6$ by the Gronwall inequality.

\subsection{Proof of Lemma \ref{lem: macro a priori}}

By taking the following velocity moments
\begin{align*}
\mu^{\frac{1}{2}}, 
v_j\mu^{\frac{1}{2}}, 
\frac{1}{6}(|v|^2-3)\mu^{\frac{1}{2}},
(v_j{v_m}-1)\mu^{\frac{1}{2}}, 
\frac{1}{10}(|v|^2-5)v_j \mu^{\frac{1}{2}}
\end{align*}
with {$1\leq j,m\leq 3$} for the first equation of \eqref{eq: linearized BE}, one sees that the coefficient functions $[a,b,c]=[a,b,c](t,x)$ satisfy the fluid-type system
\begin{align}\label{mac.law}
\begin{cases}
\partial_t a +\nabla_x b=0,\\
\partial_t b +\nabla_x (a+2c)+\nabla_x\cdot \Theta ((\mathbf{I}-\mathbf{P}) f)=0,\\
\partial_t c +\frac{1}{3}\nabla_x\cdot b +\frac{1}{6}\nabla_x\cdot \Lambda ((\mathbf{I}-\mathbf{P}) f)=0,\\
\partial_t[\Theta_{jm}((\mathbf{I}-\mathbf{P}) f)+2c\delta_{ jm}]+\partial_j b_m+\partial_m b_j=\Theta_{jm}(\mathbbm{r}+\mathbbm{h}),\\
\partial_t \Lambda_j((\mathbf{I}-\mathbf{P}) f)+\partial_j c = \Lambda_j(\mathbbm{r}+\mathbbm{h}),
\end{cases}
\end{align}
where $\Theta_{jm} (f)=((v_jv_m-1)\mu^{1/2},f)$, $\Theta(f)=(\Theta_{jm}(f))_{1\le j,m\le 3}$ and $\Lambda_i(f)=\frac{1}{10}((\vert v\vert^2-5)v_j \mu^{1/2},f)$, $\Lambda(f)=(\Lambda_j (f))_{1\le j\le 3}$ are higher-order moment functions, and
\begin{align*}
\mathbbm{r}= -v\cdot \nabla_x (\mathbf{I}-\mathbf{P})f,\ \mathbbm{h}=-L (\mathbf{I}-\mathbf{P})f+H.
\end{align*}
Notice 
\begin{align*}
\vert \Theta(f)\vert, \vert \Lambda (f)\vert\le \Vert f\Vert_{L^2_v}, \Vert f\Vert_{L^2_{v,D}},
\end{align*}
which can be shown by the method used in the proof of Lemma  \ref{lem: inner product L}.

The Fourier transform of \eqref{mac.law} is
\begin{align}\label{mac.law.fourier}
\begin{cases}
\partial_t \hat{a} +ik\cdot \hat{b}=0,\\
\partial_t \hat{b} +ik  (\hat{a}+2\hat{c})+ik\cdot \Theta ((\mathbf{I}-\mathbf{P}) \hat{f})=0,\\
\partial_t \hat{c} +\frac{1}{3}ik\cdot \hat{b} +\frac{1}{6}ik \cdot \Lambda ((\mathbf{I}-\mathbf{P}) \hat{f})=0,\\
\partial_t[\Theta_{jm}((\mathbf{I}-\mathbf{P}) \hat{f})+2\hat{c}\delta_{ jm}]+ik_j \hat{b}_m+ik_m \hat{b}_j=\Theta_{jm}(\hat{\mathbbm{r}}+\hat{\mathbbm{h}}),\\
\partial_t \Lambda_j((\mathbf{I}-\mathbf{P}) \hat{f})+ik_j \hat{c} = \Lambda_j(\hat{\mathbbm{r}}+\hat{\mathbbm{h}}).
\end{cases}
\end{align}
Hereafter \eqref{mac.law.fourier}$_n$ denotes the $n$-th equation of \eqref{mac.law.fourier}.
We first multiply the complex conjugate of $ik_j \hat{c}/(1+\vert k\vert^2)$ to \eqref{mac.law.fourier}$_5$ to have
\begin{align}
\Big(\partial_t \Lambda_j ((\mathbf{I}-\mathbf{P})\hat{f}), \frac{ik_j \hat{c}}{1+\vert k\vert^2} \Big)+\frac{k_j^2}{1+\vert k\vert^2} \vert \hat{c}\vert^2 =\Big(\Lambda_j (\hat{\mathbbm{r}}+\hat{\mathbbm{h}}), \frac{ik_j \hat{c}}{1+\vert k\vert^2} \Big).\label{ap.adp1}
\end{align}
We substitute \eqref{mac.law.fourier}$_3$ to this equation, which yields
\begin{align*}
\partial_t \Big(\Lambda_j ((\mathbf{I}-\mathbf{P})\hat{f}), \frac{ik_j \hat{c}}{1+\vert k\vert^2}\Big)
&-\Big(\Lambda_j ((\mathbf{I}-\mathbf{P})\hat{f}), \frac{k_j}{1+\vert k\vert^2}\Big[ \frac{1}{3}k\cdot \hat{b}+\frac{1}{6}k\cdot \Lambda((\mathbf{I}-\mathbf{P})\hat{f})\Big]\Big)\\
&  +\frac{k_j^2}{1+\vert k\vert^2} \vert \hat{c}\vert^2 =\Big(\Lambda_j (\hat{\mathbbm{r}}+\hat{\mathbbm{h}}), \frac{ik_j \hat{c}}{1+\vert k\vert^2}\Big).
\end{align*}
Integrating this over $[0,T]$ gives
\begin{align*}
\int^T_0 \frac{k_j^2}{1+\vert k\vert^2}\vert \hat{c}\vert^2 dt 
& \le \Big[\vert \Lambda_j ((\mathbf{I}-\mathbf{P})\hat{f}) \vert^2 + \frac{k_j^2}{(1+\vert k\vert^2)^2}\vert \hat{c}\vert^2\Big]^T_{t=0}\\
&+\varepsilon_1^2\int^T_0 \frac{\vert k\vert^2}{1+\vert k\vert^2}\vert \hat{b}\vert^2dt +\frac{1}{4\varepsilon_1^2}\int^T_0 \frac{\vert \Lambda_j((\mathbf{I}-\mathbf{P})\hat{f})\vert^2}{1+\vert k\vert^2}dt\\
&+\int^T_0 \frac{\vert k\vert^2}{1+\vert k\vert^2}\vert \Lambda ((\mathbf{I}-\mathbf{P})\hat{f})\vert^2dt \\
&+ \frac{1}{4\varepsilon_1^2}\int^T_0 \frac{\vert \Lambda_j(\hat{\mathbbm{r}}+\hat{\mathbbm{h}})\vert^2}{1+\vert k\vert^2} dt +\varepsilon_1^2\int^T_0 \frac{\vert k\vert^2}{1+\vert k\vert^2}\vert \hat{c}\vert^2 dt,
\end{align*}
where $\varepsilon_1>0$ is a small constant chosen later.
Taking the summation $j=1$, $2$, $3$ then the square root of the resultant, we have
\begin{align}\label{macro.c}
&\frac{\vert k\vert }{\sqrt{1+\vert k\vert^2}} \Big(\int^T_0 \vert \hat{c}\vert^2 dt\Big)^{1/2} 
 \le C\Vert \hat{f}_0(k)\Vert_{L^2_v}+ C\Vert  (\mathbf{I}-\mathbf{P})\hat{f} \Vert_{L^\infty_T L^2_v} + \frac{\vert k\vert}{1+\vert k\vert^2}\sup_{0\le t \le T} \vert \hat{c}\vert  \notag \\
&+C\varepsilon_1 \frac{\vert k\vert}{\sqrt{1+\vert k\vert^2}} \Big( \int^T_0 \vert \hat{b}\vert^2dt\Big)^{1/2} +\frac{C}{\varepsilon_1} \Big(\int^T_0 \frac{\vert \Lambda((\mathbf{I}-\mathbf{P})\hat{f})\vert^2}{1+\vert k\vert^2}dt \Big)^{1/2} \notag \\
&+C\frac{\vert k\vert}{\sqrt{1+\vert k\vert^2}} \Vert (\mathbf{I}-\mathbf{P})\hat{f}\Vert_{L^2_T L^2_{v,D}}  \notag \\
&+ \frac{C}{\varepsilon_1}  \Big( \int^T_0 \frac{\vert \Lambda(\hat{\mathbbm{r}}+\hat{\mathbbm{h}})\vert^2}{1+\vert k\vert^2} dt\Big)^{1/2} + C\varepsilon_1 \frac{\vert k\vert}{\sqrt{1+\vert k\vert^2}} \Big(\int^T_0 \vert \hat{c}\vert^2 dt\Big)^{1/2}.
\end{align}

Next, we consider $\hat{b}$. We multiply the complex conjugate of $(ik_j \hat{b}_m+ ik_m \hat{b}_j)/(1+\vert k\vert^2)$ to \eqref{mac.law.fourier}$_{4}$  to have
\begin{align*}
&\Big(\partial_t [\Theta_{jm} ((\mathbf{I}-\mathbf{P})\hat{f})+2\hat{c} \delta_{jm}], \frac{i}{1+\vert k\vert^2}(k_j \hat{b}_m +k_m \hat{b}_j)\Big) 
+\frac{\vert k_j \hat{b}_m +k_m \hat{b}_j\vert^2}{1+\vert k\vert^2}\\
&=\Big(\Theta(\hat{\mathbbm{r}}+\hat{\mathbbm{h}}), \frac{i}{1+\vert k\vert^2}(k_j \hat{b}_m +k_m \hat{b}_j)\Big).
\end{align*}
The substitution of  \eqref{mac.law.fourier}$_2$ gives
\begin{align*}
&\partial_t \Big( \Theta_{jm} ((\mathbf{I}-\mathbf{P})\hat{f}) +2\hat{c} \delta_{jm}, \frac{i}{1+\vert k\vert^2}(k_j \hat{b}_m +k_m \hat{b}_j)\Big)\\
& -\Big(\Theta_{jm}((\mathbf{I}-\mathbf{P})\hat{f})+2\hat{c}\delta_{jm}, \frac{1}{1+\vert k\vert^2}k_jk_m (\hat{a}+2\hat{c})\\
&+\frac{1}{1+\vert k\vert^2}[k_j\{k\cdot \Theta((\mathbf{I}-\mathbf{P})\hat{f})\}_m+k_m\{k\cdot \Theta((\mathbf{I}-\mathbf{P})\hat{f})\}_j]\Big) 
+\frac{\vert k_j \hat{b}_m +k_m \hat{b}_j\vert^2}{1+\vert k\vert^2}\\
&=\Big(\Theta(\hat{\mathbbm{r}}+\hat{\mathbbm{h}}), \frac{i}{1+\vert k\vert^2}(k_j \hat{b}_m +k_m \hat{b}_j)\Big).
\end{align*}
By the same procedure to derive \eqref{macro.c}, one may have
\begin{align}\label{macro.b}
&\frac{\vert k\vert}{\sqrt{1+\vert k\vert^2}}\Big(\int^T_0 \vert \hat{b}\vert^2dt \Big)^{1/2} 
\le C\Vert \hat{f}_0(k)\Vert_{L^2_v}+ C\Vert  (\mathbf{I}-\mathbf{P})\hat{f} \Vert_{L^\infty_T L^2_v} +\sup_t \vert \hat{c}\vert   \notag   \\
&+ \frac{\vert k\vert}{\sqrt{1+\vert k\vert^2}}\sup_t \vert \hat{b}\vert+C\frac{\vert k\vert}{\sqrt{1+\vert k\vert^2}} \Big( \Vert (\mathbf{I}-\mathbf{P})\hat{f}\Vert_{L^2_{v,D}}^2 dt\Big)^{1/2}   \notag    \\
&+\frac{C}{\varepsilon_2}\frac{1}{\sqrt{1+\vert k\vert^2}} \Big(\int^T_0 \Vert (\mathbf{I}-\mathbf{P})\hat{f}\Vert_{L^2_{v,D}}^2dt \Big)^{1/2}
+ C \varepsilon_2 \frac{\vert k\vert}{\sqrt{1+\vert k\vert^2}} \Big(\int^T_0 \vert \hat{a}\vert^2 dt\Big)^{1/2}  \notag \\
&+ (C+\varepsilon_2+ \varepsilon_2^{-1}) \frac{\vert k\vert}{\sqrt{1+\vert k\vert^2}} \Big(\int^T_0 \vert \hat{c}\vert^2dt \Big)^{1/2} \notag  \\
&+ \frac{C}{\varepsilon_2}\frac{1}{\sqrt{1+\vert k\vert^2}} \Big(\int^T_0 \vert \Theta(\hat{\mathbbm{r}}+\hat{\mathbbm{h}})\vert^2 dt \Big)^{1/2}+C\varepsilon_2 \frac{\vert k\vert}{\sqrt{1+\vert k\vert^2}}\Big(\int^T_0 \vert \hat{b}\vert^2dt \Big)^{1/2},
\end{align}
where $\varepsilon_2>0$ will be chosen later.

Third, we will consider the estimate of $\hat{a}$. \eqref{mac.law.fourier}$_2$ and \eqref{mac.law.fourier}$_1$ give
\begin{align*}
\partial_t \Big(\hat{b}, \frac{ik\hat{a}}{1+\vert k\vert^2}\Big) -\Big(\hat{b}, \frac{k\otimes k \hat{b}}{1+\vert k\vert^2}\Big) +\frac{\vert k\vert^2}{1+\vert k\vert^2}(\vert \hat{a}\vert^2+2(\hat{c},\hat{a}))
+\frac{\vert k\vert^2}{1+\vert k\vert^2}(\Theta((\mathbf{I}-\mathbf{P})\hat{f}, \hat{a})\\
=0,
\end{align*}
thus
\begin{align}\label{macro.a}
\frac{\vert k\vert}{\sqrt{1+\vert k\vert^2}}\Big( \int^T_0 \vert \hat{a}\vert^2dt\Big)^{1/2} 
&\le C\Vert \hat{f}_0(k)\Vert_{L^2_v} + \sup_t \vert \hat{b}\vert +\frac{\vert k\vert}{1+\vert k\vert^2} \sup_t \vert \hat{a}\vert  \notag \\
&+ C\frac{\vert k\vert}{\sqrt{1+\vert k\vert^2}} \Big(\int^T_0 \vert\hat{b}\vert^2+\vert \hat{c}\vert^2 dt \Big)^{1/2} \notag \\
& +C \frac{\vert k\vert}{\sqrt{1+\vert k\vert^2}} \Big(\int^T_0 \Vert (\mathbf{I}-\mathbf{P})\hat{f}\Vert^2_{L^2_{v,D}} dt\Big)^{1/2}.
\end{align}

Now, taking the summation \eqref{macro.c}$+\delta_2$\eqref{macro.b}$+\delta_3$\eqref{macro.a} with $0< \delta_3\ll \delta_2\ll 1$ then taking $\varepsilon_1$ and $\varepsilon_2$ sufficiently small give
\begin{align*}
\sup_k \frac{\vert k\vert }{\sqrt{1+\vert k\vert^2}} \Big(\int^T_0 \vert (\hat{a},\hat{b},\hat{c})\vert^2 dt\Big)^{1/2} 
&\le C\Vert f_0\Vert_{L^\infty_k L^2_v} +C \Vert f\Vert_{L^\infty_k L^\infty_T L^2_v} \\
&+C \sup_k \Big(\int^T_0 \Vert (\mathbf{I}-\mathbf{P})\hat{f}\Vert_{L^2_{v,D}}^2 dt\Big)^{1/2} \\
&+C\sup_k   \Big(\int^T_0 \frac{\vert \Lambda(\hat{\mathbbm{r}}+\hat{\mathbbm{h}})\vert^2}{1+\vert k\vert^2}dt \Big)^{1/2}\\
&+ C\sup_k  \Big(\int^T_0 \frac{\vert \Theta(\hat{\mathbbm{r}}+\hat{\mathbbm{h}})\vert^2}{1+\vert k\vert^2}dt \Big)^{1/2}.
\end{align*}
By Lemma \ref{lem: inner product L}, it is easy to derive
\begin{align*}
&\sup_k\Big[\Big(\int^T_0 \frac{\vert \Lambda(\hat{\mathbbm{r}}+\hat{\mathbbm{h}})\vert^2}{1+\vert k\vert^2}dt \Big)^{1/2},
\Big(\int^T_0 \frac{\vert \Theta(\hat{\mathbbm{r}}+\hat{\mathbbm{h}})\vert^2}{1+\vert k\vert^2}dt \Big)^{1/2} \Big] \\
&\le C \Vert (\mathbf{I}-\mathbf{P})f\Vert_{L^\infty_k L^2_T L^2_{v,D}}+\sup_k \Big(\int^T_0 \frac{\vert (\hat{H}, \mu^{1/4})_{L^2_v})\vert^2}{1+\vert k\vert^2}dt \Big)^{1/2}.
\end{align*}
This finally gives the desired estimate.\qed

\subsection*{Data Availability Statement}
All data generated or analysed during this study are included in this published article.

\subsection*{Acknowledgements}
RJD is supported by the General Research Fund (Project No.~14301720) from RGC of Hong Kong. 
SS is supported by JSPS Kakenhi Grant (No.\ 20K14338).
YU is supported by JSPS Kakenhi Grant (No.\ 21K03327).

\frenchspacing
\bibliographystyle{cpam}

\end{document}